\documentclass[reqno]{amsart}
\usepackage{comment}
\usepackage{fancyhdr}
\usepackage{tikz}
\usepackage{relsize}
\usepackage{booktabs}
\usepackage{amssymb, latexsym, amsthm, enumitem, color, amsmath}
\usepackage{array}
\usepackage{mathtools}
\usepackage{nicefrac}
\usepackage{upquote}
\usepackage{graphicx}
\allowdisplaybreaks[0]
\usepackage[unicode=true,pdfusetitle,bookmarks=true,bookmarksnumbered=false,
bookmarksopen=false,breaklinks=true,pdfborder={0 0 0},
pdfborderstyle={},backref=false,colorlinks=true]{hyperref}
\definecolor{myblue}{rgb}{0.09,0.32,0.44} 
\hypersetup{pdfborder={0 0 0},pdfborderstyle={},colorlinks=true,linkcolor=myblue,citecolor=myblue,urlcolor=blue}

\newtheorem{thm}{Theorem}[section] 
\newtheorem*{thm*}{Theorem}

\newtheorem{claim}[thm]{Claim}

\newtheorem{cor}[thm]{Corollary}

\newtheorem{fact}[thm]{Fact}
\newtheorem{lem}[thm]{Lemma}
\newtheorem{prop}[thm]{Proposition}

\newtheorem{question}[thm]{Question}
\newtheorem*{que*}{Question}
\newtheorem*{fact*}{Fact}

\newtheorem{rem}[thm]{Remark}

\newtheorem*{rem*}{Remark}
\newtheorem*{rems*}{Remarks}

\newcommand\Cref[1]{{Corollary~\ref{#1}}}

\newcommand{\N}{\mathbb{N}}
\newcommand{\Nstar}{\mathbb N_{>0}}
\newcommand{\Z}{\mathbb{Z}}
\newcommand{\R}{\mathbb{R}}

\newcommand{\E}{\mathbb{E}}
\newcommand{\eps}{\varepsilon}

\newcommand{\argmin}{\mathrm{argmin}}

\newcommand{\Diam}{\mathrm{Diam}}
\newcommand{\en}{\mathcal{E}}

\newcommand{\hld}{H\"older}

\renewcommand{\Pr}{\mathbb{P}}
\newcommand{\sign}{\mathrm{sgn}}

\makeatletter
\def\moverlay{\mathpalette\mov@rlay}
\def\mov@rlay#1#2{\leavevmode\vtop{%
   \baselineskip\z@skip \lineskiplimit-\maxdimen
   \ialign{\hfil$\m@th#1##$\hfil\cr#2\crcr}}}
\newcommand{\charfusion}[3][\mathord]{
    #1{\ifx#1\mathop\vphantom{#2}\fi
        \mathpalette\mov@rlay{#2\cr#3}
      }
    \ifx#1\mathop\expandafter\displaylimits\fi}
\makeatother

\newif
\iffurther 
\furtherfalse

\makeatletter
\def\@settitle{\begin{center}%
  \baselineskip14\p@\relax
    \normalfont\LARGE
  \@title
  \end{center}%
}
\makeatother

 \title{Sharp polynomial convergence rates for \texorpdfstring{$\ell^p$}{ell-p}-relaxation on graphs with boundary}
\author{Chenyu Gan}

\address{C. Gan, Qiuzhen College, Tsinghua University,   Beijing, China}
\email{gancy22@mails.tsinghua.edu.cn }

\author{Yuval Peres}

\address{Y. Peres, Beijing Institute of Math.\ Sciences and Applications (BIMSA), Huairou district, Beijing, China}
\email{yperes@gmail.com}

\author{Junchi Zuo}
\address{J. Zuo, Qiuzhen College, Tsinghua University,   Beijing, China}
\email{zjczzzjc@126.com}

\begin{document}

\begin{abstract}
We study asynchronous $\ell^p$-energy minimization on finite connected graphs with prescribed boundary values. At each step, a uniformly chosen interior vertex updates its value to minimize the local energy. For initial values in $[0,1]$, we bound the expected time to approach the $p$-harmonic extension within $\epsilon$ in supremum norm, for $0<\epsilon\le1/2$. On graphs with $n$ vertices, the bound is $C_p n^{\beta_p}\epsilon^{-(2-p)/(p-1)}$ for $1<p<2$ and $C_p n^{\beta_p}\log(n/\epsilon)$ for $p\ge2$, where $C_p$ depends only on $p$ and $\beta_p=\max\{2p/(p-1),3\}$. Bounds incorporating average degree are jointly sharp in graph size, degree, and accuracy for $1<p<2$; for $p\ge2$, lower bounds match the polynomial dependence on size and degree at fixed accuracy. We also establish corresponding synchronous bounds and graph-dependent estimates for $p=2$.

\end{abstract}
\maketitle

\section{Introduction}
Let $G=(V,E)$ be a finite connected graph\footnote{All graphs considered in this paper are
undirected and simple (with no self-loops or multiple edges).}, where each vertex $v \in V$ is assigned an initial opinion $f_0(v)\in[0,1]$. Let $V = I \bigsqcup B$ (a disjoint union), where $I, B \neq \varnothing$.
Given $1<p<\infty$, the asynchronous $\ell^p$\textit{-relaxation dynamics} on $G$ are defined by choosing a vertex $v_t\in I$ uniformly, independently at random in each step $t \ge 1$, and updating the value at $v_t$ to minimize the $\ell^p$-energy of $f_t$, leaving the values at other vertices unchanged
\begin{align}\label{eq:pdef}
    f_t(v) = \begin{cases}
        \argmin_{y} \sum_{w\sim v} \lvert f_{t-1}(w)-y \rvert^p \, \quad &v=v_t, \\
        f_{t-1}(v)  \quad &v\neq v_t .
    \end{cases}
\end{align} 
The minimizing $y$ is unique, since  the function $x \mapsto \lvert x \rvert^p$ is strictly convex on $\R$.
These dynamics are linear when $p=2$; more generally, for every $p>1$ they are linear whenever every interior vertex has degree at most~2.  For $p\ne2$, updates at interior vertices of degree at least~3 are generally nonlinear.
The following fact is well known and follows from Perron's method; see
Section~\ref{sub: perron method}.
\begin{fact}\label{fa:conv to cons}
    For every $p \in (1, \infty)$, any finite connected graph $G=(V,E)$, and every initial profile $f_0$, we have $\lim_{t \to \infty} f_t(v)=h(v)$ almost surely, 
     where $h$ is the unique $p$-harmonic extension of $f_0 \vert_{B}$ defined in Section~\ref{sub: perron method}. 
\end{fact}

\begin{que*}
    How fast do the opinions converge?
\end{que*}

We measure the distance between $f_t$ and $h$ using the $\ell^{\infty}$ norm,  
\begin{align}
    \lVert f_t - h \rVert_\infty := \max_{v \in V} \lvert f_t(v) - h(v) \rvert.
\end{align}
The {\bf $\epsilon$-approximation time} is defined by:
\[ \tau_p(\epsilon) := \min\{t\ge 0 :\, \lVert f_t - h \rVert_\infty \le \epsilon \}.\]

We also define the {\bf $\ell^p$-energy} of any function $f: V \to \R$ to be 

\begin{align}\label{eq:energy def}
    \en_p(f):= \sum_{\{v,w\} \in E}  \lvert f(v)-f(w) \rvert^p .
\end{align}

The recursive definition of $ f_t $  implies  that $\en_p(f_t) \le \en_p(f_{t-1})$  for all $t \ge 1$. 

\subsection{\texorpdfstring{Bounds for expected $\epsilon$-approximation time as a function of $n$ and $\epsilon$}{Bounds for expected epsilon-approximation time as a function of n and epsilon}}

\begin{thm}\label{th: lp with boundary simplified}
    Fix $p \in (1, \infty)$ and define
    $\beta_p=\max\Bigl\{\frac{2p}{p-1}, 3\Bigr\}$.
    There exists a constant $C_p > 0$ 
    such that for every $n \ge 2$, any connected graph $(V,E)$ with $\lvert V \rvert=n$, any decomposition $V = I \bigsqcup B $ such that $I, B \neq \varnothing$, any initial profile $f_0:V \to [0,1]$ and any $\epsilon \in (0,1/2]$, the dynamics \eqref{eq:pdef} satisfy
\smallskip    
    \begin{itemize}
        \item[{\bf(a)}] If $p \in (1, 2)$, then $\E[\tau_p(\epsilon)] \le C_pn^{\beta_p} \epsilon^{-\frac{2-p}{p-1}}$. 
  \smallskip    
        \item[{\bf(b)}] If $p \in [2, \infty)$, then $\E[\tau_p(\epsilon)]\le C_p n^{\beta_p} \log(n/\epsilon) .
    \;$ 
    \medskip 
    
    \end{itemize}
    Moreover, there exist constants $C_p^*,c_p>0$ for which the following statements hold for every $N\ge C_p^*$.
    If $p\in(1,2)$, then for every $\epsilon\in(0,1/2]$ there is a connected graph
    $G=(V=I\bigsqcup B,E)$ with $\lvert V\rvert\le N$ and an initial profile
    $f_0:V\to[0,1]$ such that
\begin{align}\label{eq:converse main}
    \tau_p(\epsilon)\ge
    c_pN^{\beta_p}\epsilon^{-\frac{2-p}{p-1}}.
\end{align}
    If $p\in[2,\infty)$, there is a connected graph
    $G=(V=I\bigsqcup B,E)$ with $\lvert V\rvert\le N$ and an initial profile
    $f_0:V\to[0,1]$ such that
\begin{align}
    \tau_p(1/2)\ge c_pN^{\beta_p}.
\end{align}
    Both lower bounds hold for every sequence of update vertices in $I$.
\end{thm}



For $p\in(1,2)$, the upper and lower bounds in Theorem~\ref{th: lp with boundary simplified} match simultaneously in $N$ and $\epsilon$.  For $p\ge2$ and fixed $\epsilon$, they match up to the logarithmic factor in the upper bound.  At a fixed error threshold, the polynomial exponent $\beta_p$ agrees with the sharp exponent for consensus without boundary values obtained in~\cite{noboundarycase}.

In~\cite[Question~6.2]{noboundarycase}, with $n$ denoting the number of interior vertices, Amir, Nazarov and Peres asked for the worst-case asymptotics of the approximation time with prescribed boundary values and, for $p>2$, conjectured the exponent $\beta_p$.  Our notation $n=\lvert V\rvert$ counts all vertices.  Thus, in the regime $\lvert B\rvert=O(\lvert I\rvert)$, we have $n=\Theta(\lvert I\rvert)$, and Theorem~\ref{th: lp with boundary simplified} verifies their conjectured polynomial exponent, with the logarithmic factor displayed above.  We do not claim an $\lvert I\rvert$-only bound when the boundary can be arbitrarily larger than the interior.

Proposition~\ref{powerlower} isolates the sharp dependence on $\epsilon$ for $p\in(1,2)$ on a four-vertex graph.  The joint lower bound in Theorem~\ref{th: joint lower p<2} combines this slow small-error mechanism with the graph-size obstruction.

Theorem~\ref{th: lp with boundary simplified} will follow from Theorem~\ref{th:main lpbound}, which gives refined upper bounds on the $\epsilon$-approximation time in terms of the number of vertices $\lvert V \rvert$ and the average degree.

\subsection{Sharp bounds using average degree}\label{s:detailed results}
Let $G=(V,E)$ be a connected graph with $\lvert V \rvert =n$. Next, we will give a more precise upper bound for $\E[\tau_p(\epsilon)]$   in terms of $n$ and the 
 {\em average degree} $D = D_G:=\frac{1}{n}\sum_{v\in V}{\deg}(v) \, $.  Recall that $\beta_p=\max\Bigl\{\frac{2p}{p-1},3\Bigr\}$
  and define
\begin{align}\label{eq:thetadef}
        \theta_p:=\frac{1}{p-1} \quad \text{\rm if} \quad  p \in (1, 2] \quad \text{\rm and} \quad \theta_p:=\max\Bigl\{\frac{3-p}{p-1},0 \Bigr\} \quad \text{\rm if} \quad  p \in (2, \infty) \,.
    \end{align}  
    Consider the function  
\begin{align} \label{defF}
F_*(n,p,D):=
 \begin{cases}
n^\frac{1-2p}{p-1}D^\frac{-1}{p-1} \, \ \, &p \in (1, 2) ,\\
n^{-3}D^\frac{p-3}{p-1} \, \ \, & p \in [2, 3) , \\ n^{-3} \, \ \, &p \in [3, \infty) .    \end{cases}
\end{align}
 
Note that $F_{*}(n, p, D) = n^{-\beta_p}(\frac{D}{n})^{-\theta_p}$ for $p \in (1, \infty)$. 

\begin{thm}\label{th:main lpbound}
    Fix $1<p<\infty$. There exists a constant  $C_p>0$ such that for every $n \ge 2$,  any connected graph $G=(V,E)$ with $\lvert V \rvert = n$ and average degree $D$, any decomposition $V = I \bigsqcup B$ such that $I, B \neq \varnothing$, 
    any initial profile $f_0:V \to [0,1]$ and any $\epsilon \in (0, 1/2]$, the dynamics \eqref{eq:pdef} satisfy
    \smallskip  
        \begin{itemize}
        \item[{\bf(a)}] If $p \in (1, 2)$, then $\E[\tau_p(\epsilon)] \le \frac{C_p\epsilon^{-\frac{2-p}{p-1}}}{F_{*}(n, p, D)}$. 
    \smallskip    
        \item[{{\bf(b)}}] If $p \in [2, \infty)$, then $\E[\tau_p(\epsilon)]\le \frac{C_p \log(n/\epsilon)}{F_{*}(n, p, D)}  
    .$
    \end{itemize}
    \medskip
Moreover, after increasing $C_p$ if necessary, there is a constant $c_p>0$ for which the following statements hold for every $N\ge C_p$ and every $2\le D\le N$.
If $p\in(1,2)$, then for every $\epsilon\in(0,1/2]$ there is a connected graph
$G=(V=I\bigsqcup B,E)$ with $\lvert V\rvert\le N$ and $D_G\le D$, and an initial profile $f_0:V\to[0,1]$, such that
\begin{align}\label{eq:p D lower}
\tau_p(\epsilon) \ge
c_p F_*(N,p,D)^{-1}
\epsilon^{-\frac{2-p}{p-1}}.
\end{align}
If $p\in[2,\infty)$, there is a connected graph $G=(V=I\bigsqcup B,E)$ with $\lvert V\rvert\le N$ and $D_G\le D$, and an initial profile $f_0:V\to[0,1]$, such that
\begin{align}
\tau_p(1/2) \ge
c_p F_*(N,p,D)^{-1}.
\end{align}
Both lower bounds hold for every sequence of update vertices in $I$.

\end{thm}
Theorem~\ref{th: lp with boundary simplified} follows by using $D_G\le\lvert V\rvert$ in the upper bounds and by taking $D=N$ in the lower bounds.

Parts (a) and (b) of Theorem~\ref{th:main lpbound} are proved in Section~\ref{s:lpbounds}.  The lower bound for $p\in(1,2)$ is Theorem~\ref{th: joint lower p<2}; the fixed-threshold lower bound for $p\ge2$ follows from Theorem~\ref{th: with boundary lower bound}.

Assuming the result of Theorem~\ref{th:main lpbound}, we give a refined upper bound for expected $\epsilon$-approximation time in the case $f_0 \vert_{B}$ is constant. 

\begin{cor}\label{cor: constant boundary data}
    In the setting of Theorem~\ref{th:main lpbound}, if we furthermore assume that $f_0 \vert_{B}$ is a constant, then the upper bound for $p\in (1,2)$ can be improved to $\frac{C_p \log(1/\epsilon)}{F_{*}(n, p, D)}$ for some constant $C_p > 0$. 
\end{cor}
\begin{proof}
    We write $\E_{f_0}$ to emphasize the dependence on the initial profile and, just for this proof, put $T_k:=\tau_p(2^{-k})$.  At the stopping time $T_2$, the profile $f_{T_2}$ is $\mathcal F_{T_2}$-measurable.  Since its $p$-harmonic extension $h$ is constant and $\lVert f_{T_2}-h\rVert_\infty\le1/4$, the affine rescaling
    \[
        \widetilde f_0:=2f_{T_2}-2\min_V f_{T_2}
    \]
    is $\mathcal F_{T_2}$-measurable, takes values in $[0,1]$, and is constant on $B$.  The update rule commutes with positive affine transformations.  By the strong Markov property of the i.i.d. update sequence, conditionally on $\mathcal F_{T_2}$ the future updates have the law of a fresh process started from $\widetilde f_0$.  Therefore, for $k\ge2$,
    \[
        \E_{f_0}[T_k-T_2\mid\mathcal F_{T_2}]
        \le \sup_{\widetilde f_0}\E_{\widetilde f_0}[T_{k-1}],
    \]
    where the supremum is over profiles in $[0,1]^V$ that are constant on $B$.  Taking expectations gives
    \[
        \E_{f_0}[T_k]
        \le \E_{f_0}[T_2]+\sup_{\widetilde f_0}\E_{\widetilde f_0}[T_{k-1}].
    \]
    Induction yields
    \[
        \sup_{f_0}\E_{f_0}[T_k]
        \le k\sup_{f_0}\E_{f_0}[T_2],
    \]
    with the same class of initial profiles.  The corollary now follows from Theorem~\ref{th:main lpbound}{\rm(a)}.
\end{proof}

We also analyze the {\bf synchronous version} of the $\ell^p$-relaxation dynamics. Given an initial profile $f_0^*:V \to \R$, we inductively define functions $f_t^{*}: V \to \R$ for all $t \in \Nstar$ by setting  
\begin{align}\label{eq:synch}
    f_t^*(v):=\argmin_{y} \sum_{w\sim v}\lvert f_{t-1}^*(w)-y \rvert ^p \, \quad \forall v\in I , 
\end{align}

and $f_t^*(v) := f_{t-1}^*(v)$ for $v \in B$. 

We also define the $\eps$-approximation time with respect to the  synchronous version of the $\ell^p$-relaxation dynamics by
\begin{align}
    \tau_p^{sync}(\epsilon) := \min\{t\ge 0 :\, \lVert f_t^* - h \rVert_\infty \le \epsilon \},
\end{align}

where $h$ is the unique $p$-harmonic extension of $f_0^*\vert_B$.

Intuitively, since the synchronous dynamics update all vertices at the same time, we expect that the convergence rate of the synchronous dynamics should be about $\lvert I \rvert$ times faster than the asynchronous dynamics.
In fact, we can obtain upper and lower bounds for the $\epsilon$-approximation time of the synchronous dynamics from the proof of Theorem~\ref{th:main lpbound}:
\begin{cor}\label{cor: main cor}
    Fix $1<p<\infty$. There exists a constant  $C_p>0$ such that for every $n \ge 2$, any connected graph $G=(V,E)$ with $\lvert V \rvert = n$ and average degree $D$, any decomposition $V = I \bigsqcup B$ such that $I, B \neq \varnothing$, 
    any initial profile $f_0^*:V \to [0,1]$ and any $\epsilon \in (0, 1/2]$, the dynamics \eqref{eq:synch} satisfy
    \smallskip  
        \begin{itemize}
        \item[{\bf (a)}]  If $p \in (1, 2)$, then $\tau_p^{sync}(\epsilon) \le \frac{C_p\epsilon^{-\frac{2-p}{p-1}}}{\lvert I\rvert F_{*}(n, p, D)}  .$
    \smallskip    
        \item[{\bf(b)}] If  $p \in [2, \infty)$, then $\tau_p^{sync}(\epsilon)\le \frac{C_p \log(n/\epsilon)}{\lvert I\rvert F_{*}(n, p, D)}  .$ 
    \end{itemize}
    \medskip
Moreover, after increasing $C_p$ if necessary, there is a constant $c_p>0$ for which the following statements hold for every $N\ge C_p$ and every $2\le D\le N$.  If $p\in(1,2)$, then for every $\epsilon\in(0,1/2]$ there is a connected graph $G=(V=I\bigsqcup B,E)$ with $\lvert V\rvert\le N$ and $D_G\le D$, and an initial profile $f_0^*:V\to[0,1]$, such that
\begin{align}
\tau_p^{sync}(\epsilon) \ge
c_p
\frac{\epsilon^{-\frac{2-p}{p-1}}}{{\lvert I\rvert F_*(N, p, D)}}.
\end{align}
If $p\in[2,\infty)$, there is a connected graph $G=(V=I\bigsqcup B,E)$ with $\lvert V\rvert\le N$ and $D_G\le D$, and an initial profile $f_0^*:V\to[0,1]$, such that
\begin{align}
\tau_p^{sync}(1/2) \ge
c_p \frac{1}{\lvert I\rvert F_*(N, p, D)}.
\end{align}
\end{cor}

The upper bound of Corollary~\ref{cor: main cor} is proved in Section~\ref{sec:synchronous-upper} and the lower bound is proved in Section~\ref{sec:lowerbounds}.

\subsection{\texorpdfstring{Graph-dependent convergence rates for $p=2$}{Graph-dependent convergence rates for p=2}}\label{s:graph-depend}
For $p \ne 2$, it is an open problem to determine the expected approximation time $\E[\tau_p(\epsilon)]$,
maximized over initial profiles $f_0:V \to [0,1]$,
as a function of the graph $G=(V,E)$ and the boundary $B$. For $p = 2$, we give graph-dependent upper and lower bounds in terms of spectral and hitting-time parameters. At a fixed error threshold, our bounds in terms of the maximal mean hitting time agree within a logarithmic factor.
Let $V=I \bigsqcup B$.

The $\ell^2$-relaxation dynamics are given by
\begin{align} \label{2def}
    f_t(v) = \begin{cases}
        \frac{1}{\text{deg}(v)}\sum_{w\sim v}f_{t-1}(w) \quad &v=v_t , \\
        f_{t-1}(v)  \quad &v\neq v_t ,
    \end{cases}
\end{align}
where $\{v_t\}_{t \ge 1}$ are chosen independently and uniformly from $I$. This is called the asynchronous DeGroot dynamics in 
~\cite{elboim2022asynchronous}. That paper relates the expected approximate consensus time (when the boundary is empty) on a graph $G$, maximized over the initial profile $f_0: V \to [0, 1]$,   to the  {\bf mixing time} of the graph.  The next proposition relates the expected approximation time when the boundary $B$ is nonempty to the maximal mean {\bf hitting time} of $B$  
\begin{equation} \label{defhit}
    t_B:=\max_{v\in I}\E_v[\tau_B] ,
\end{equation}
where $\tau_B$ is the hitting time of $B$ and $\E_v$ represents expectation for a simple random walk started at $v$. (Recall that on a hypercube $\{0,1\}^k$, the mixing time of lazy simple random walk is  $\Theta(k \log k)$, while the maximal mean hitting time of a fixed vertex is $\Theta(2^k)$).

\begin{prop}\label{graph dependent bound: hitting time}
     For every $n \ge 2$, any connected graph $(V,E)$ with $\lvert V \rvert =n$, any decomposition $V = I \bigsqcup B $ such that $I, B \neq \varnothing$ and $\lvert I \rvert \ge 2$,  any initial profile $f_0:V \to [0, 1]$ and any $\epsilon \in (0,1)$, the dynamics \eqref{eq:pdef} satisfy 
    \begin{equation} \label{hitbound}
    \E[\tau_2(\eps)] \le 4t_{B}\lvert I \rvert\log(n/\epsilon) \,.
    \end{equation}
     Conversely, let $f_0:V\to[0,1]$ equal $1$ on $I$ and $0$ on $B$. Then, for every $\epsilon\in(0,1)$,
    \begin{equation}\label{hitbound: converse}
      \E_{f_0}[\tau_2(\epsilon)]
      \ge \max\left\{\frac{(1-\epsilon)t_B\lvert I\rvert}{2},
      \ \lvert I\rvert H_{\lvert I\rvert}\right\}.
    \end{equation}
    In particular,
    \[
      \sup_{f_0:V\to[0,1]}\E_{f_0}[\tau_2(1/2)]
      \ge \max\left\{\frac{t_B\lvert I\rvert}{4},
      \ \lvert I\rvert H_{\lvert I\rvert}\right\},
    \]
    where $H_m:=\sum_{j=1}^m j^{-1}$ is the $m$th harmonic number.
\end{prop}

Proposition~\ref{graph dependent bound: hitting time} above is proved in Section~\ref{p=2 case}.

\subsection{Background and history}\label{s:background}

The $\ell^p$\textit{-relaxation dynamics} without boundary (i.e., the case $B = \varnothing$) were recently analyzed in~\cite{noboundarycase}.   In this case, $f_t$ converges to a constant function almost surely, i.e., there exists a number $x \in [0, 1]$ such that \[ \lim_{t \to \infty} f_t(v) = x \quad \forall v \in V \quad a.s.\]

In order to quantify the speed of consensus, they consider the following random variable called $\epsilon$-consensus time: 

\[ \tau_p^{\rm cons}(\epsilon) = \inf \{ t \ge 0: \max_{v \in V} f_t(v) - \min_{v \in V} f_t(v) \le \epsilon\} .\]

The lower-bound constructions in~\cite{noboundarycase} show that, at a fixed error threshold, the worst-case consensus time has order at least $n^{\beta_p}(D/n)^{\theta_p}$.  By pinning a vertex of minimum initial opinion, Theorem~\ref{th: with boundary lower bound} transfers those constructions to the boundary setting.  Thus the polynomial dependence on $n$ and $D$ in Theorem~\ref{th:main lpbound} is necessary.  For $1<p<2$, Theorem~\ref{th: joint lower p<2} additionally shows that this dependence is jointly sharp with the factor $\epsilon^{(p-2)/(p-1)}$.

A key ingredient in our upper-bound argument is a graph
$\ell^p$-Poincar\'e estimate; see
Lemma~\ref{Lemma: poincare inequality}. For broader background on $L^p$-Poincar\'e inequalities on graphs, see Levi, Santagati, Tabacco,
and Vallarino~\cite{Levi_2023}.

The case $p = 2$ without boundary is an asynchronous version of the dynamics introduced by DeGroot~\cite{DeGroot1974reaching}. 
Tight bounds on expectation of $\epsilon$-consensus time for the asynchronous version of the DeGroot dynamics were recently proved by Elboim, Peres and Peretz in~\cite{elboim2022asynchronous}, and are related to the spectral properties of the graph. Converting their continuous-time normalization, in which the total update rate is $n$, Theorem $2.1$ (b) there implies that \[ \E[\tau_2^{\rm cons}(\epsilon)]\le 8n \lvert E \rvert \cdot  \Diam(G) \lceil \log_2(1/\epsilon)\rceil \le 4n^3D \lceil \log_2(1/\epsilon)\rceil ,\]



The dynamics can also be considered in continuous time, as was done for $p=2$  in~\cite{elboim2022asynchronous}: 
Equip all the vertices in $I$ with i.i.d.\ Poisson clocks of rate $1$, and update the value at a vertex when its clock rings. By rescaling,   we can easily rewrite our result in the continuous time setting. In particular, if $\tau^{\rm Cont}_p(\epsilon)$ denotes the $\epsilon$-approximation time of the continuous-time dynamics, then $\E\bigl[\tau^{\rm Cont}_p(\epsilon)\bigr]=\E\bigl[\tau_p(\epsilon)\bigr]/\lvert I \rvert$.

The dynamics can also be regarded as an algorithm for finding a minimizer of the $\ell^p$-energy of $f$ given the boundary value $f \vert_{B} = \psi$, which can be reduced to an $\ell^p$-regression problem. 

The  problem of minimizing the $\ell^p$-energy of $f:V \to \R$ with prescribed boundary conditions $f\vert_B=\psi$ can be expressed as
\begin{align} \label{l^p energy minimizing problem}
    \min_{f: V \to \R, f \vert_{B} = \psi} \en_p(f) .
\end{align}

Assign each edge  in $E$ an arbitrary orientation, and let $M \in \R^{\lvert E \rvert \times n}$ be the directed edge-vertex incidence matrix of $G$. By relabeling vertices in $V$, we can represent $f: V \to \R$ by 
\begin{align}
    f = \begin{pmatrix}
        f\vert_{I} \\ f\vert_{B}
    \end{pmatrix} .
\end{align}

Then problem \eqref{l^p energy minimizing problem} can be rephrased as  
\begin{align}
    \min_{f \in \R^n, f\vert_{B} = \psi} \lVert M f \rVert_p .
\end{align}

Write $M = \begin{pmatrix} A & A' \end{pmatrix}$, where $A \in \R^{\lvert E \rvert \times \lvert I \rvert }$ represents the first $\lvert I \rvert$ columns of $M$ and $A' \in \R^{\lvert E \rvert \times \lvert B \rvert}$ consists of the columns
corresponding to boundary vertices. Taking  $b := -A'\psi$, we can rewrite problem \eqref{l^p energy minimizing problem}  in the form
\begin{align} \label{p-regress} 
    \min_{x \in \R^{\lvert I \rvert}} \lVert Ax - b \rVert_p .
\end{align}

The algorithmic literature on $\ell^p$-regression includes sampling
and coreset methods for overconstrained regression developed by Dasgupta, Drineas, Harb, Kumar, and Mahoney~\cite{dasgupta2009samplingalgorithmscoresetslp}, as well as
iterative-refinement methods for high-accuracy $\ell^p$-regression
developed by Adil, Kyng, Peng, and
Sachdeva~\cite{adil2019iterativerefinementellpnormregression}.
Jambulapati, Liu, and
Sidford~\cite{jambulapati2021improvediterationcomplexitiesoverconstrained}
gave improved iteration complexities for overconstrained $p$-norm
regression, including algorithms for obtaining $(1+\epsilon)$-approximate solutions to problems closely related
to \eqref{p-regress}.

In~\cite{adil2023fastalgorithmsellpregression}, Adil, Kyng, Peng and Sachdeva study a general form of $\ell^p$-regression. They consider the optimization of the following form
\begin{align} \label{genp-regress}
    \min_{Ax = b}  c^Tx + \lVert Mx \rVert_2^2 + \lVert Nx \rVert_p^p ,
\end{align}
where $c \in \R^n$, $A \in \R^{d \times n}$, $M \in \R^{m_1 \times n}$, $N \in \R^{m_2 \times n}$, $m_1, m_2 \ge n$ and $d \le n$. 

While not our main focus, the  $\ell^p$-relaxation dynamics can be viewed as a simple and natural algorithm for the $\ell^p$-regression problem \eqref{p-regress}, and we obtain tight polynomial  
bounds for the mean number of coordinate updates needed for an $\epsilon$-approximation in $\ell^\infty$-distance to the unique minimizer. For a more general $\ell^p$-regression problem, the natural analog would be to iteratively choose a variable at random and solve the resulting one-dimensional optimization problem. 
  
Our results are also  relevant to the  study of
graph-based semi-supervised learning. Slep\v{c}ev and
Thorpe~\cite{slepcev2017analysisplaplacianregularizationsemisupervised} analyzed $p$-Laplacian regularization on weighted graphs for assigning
labels to unlabeled data. Flores, Calder, and Lerman~\cite{flores2022analysis} developed theoretical results and
algorithms for $\ell_p$-based semi-supervised learning on graphs,
including graph $p$-Laplacian regularization. More recently, Robertson,
Holtz, Wan, Mishne, and
Cloninger~\cite{robertson2025robustgraphbasedsemisupervisedlearning} introduced a robust graph-based semi-supervised learning framework based
on $p$-conductances.

\noindent{\bf Acknowledgement of AI use.}
Almost all the results of this paper were obtained by the authors without use of AI and were recorded in the first arxiv version \mbox{https://arxiv.org/abs/2512.05424v1} in December of 2025.
To prepare the current manuscript, we used ChatGPT 6 to obtain a logarithmic improvement for $p=3$ and improve the dependence on $\epsilon$ in the lower bound for $p \in (1,2)$ in Theorem \ref{th:main lpbound} and Corollary \ref{cor: main cor}.

\section{Preliminaries}\label{s:preliminaries}
Throughout the paper, constants indexed by $p$ may depend on $p$ but on no graph or accuracy parameter.  Constants carrying additional numerical subscripts (such as $c_{1,p}$ or $C_{2,p}$) are local to the proof in which they are introduced unless an explicit cross-reference says otherwise.
\subsection{\texorpdfstring{Basic properties of the $p$-Laplacian}{Basic properties of the p-Laplacian}}\label{s:monotone}

Let $f:V \to \R$ and fix an interior vertex $v_1 \in I$. For $p \in (1, \infty)$, the dynamics \eqref{eq:pdef}
at time $t=1$, with initial opinion profile $f_0=f$ and update vertex $v=v_1 \in I$, yield a new value $f_1(v)$ that minimizes the local energy at $v$, 
\[L_f^v(y)=\sum_{w \sim v} \lvert y-f(w) \rvert^p \,.\] 
Thus

\[ 0=\frac{d}{dy}L_f^v(f_1(v)) = p\sum_{w \sim v} U(f_1(v) - f(w)) ,
\]

where $U(\cdot)$ is defined by
\begin{align} \label{eq:defU} U(x) := \sign(x) \lvert x \rvert^{p - 1} \quad \forall x \in \R .
\end{align}

For any $v \in I$ and any function $f: V \to \R$, we define
\begin{align}
    R_p f(v) := f(v) - m_v(f) , 
\end{align}

where $m_v(f) := \argmin_{y \in \R} \sum_{w \sim v} \lvert y - f(w) \rvert^p$. 

For convenience, we set $R_p f(v) := 0$ for $v \in B$. 

We say that the function $f:V \to \R$ is $p$-{\bf harmonic}  at $v \in I$ if $R_p f(v) = 0$. We say that $f$ is  $p$-{\bf superharmonic} (respectively, $p$-{\bf subharmonic}) at $v \in I$ if it satisfies $R_p f(v) \ge 0$ (respectively, $R_p f(v) \le 0$).   

The condition that  $R_p f(v) = 0$ is equivalent to  
\begin{align} \label{eq:pharmonic}
0 = \frac{d}{dy}L_f^v(f(v))=p\sum_{w \sim v} U\big(f(v) - f(w)\big) ,
\end{align}
while $R_p f(v) \ge 0$ is equivalent to 
\begin{align} \label{eq:psuper}
0 \le \frac{d}{dy} L_f^v (f(v))=p\sum_{w \sim v} U\big(f(v) - f(w)\big) . 
\end{align}

Similarly, $R_p f(v) \le 0$ is equivalent to $0 \ge \frac{d}{dy} L_f^v(f(v))$.

The following claim shows that $p$-harmonic functions minimize $\ell^p$-energy among functions with the same boundary values.

\begin{claim} \label{energy minimizer}
    Suppose $h: V \to \R$ is $p$-harmonic in $I$. Then for any function $f: V \to \R$ such that $f \vert_B = h \vert_{B}$, we have $\en_p(f) \ge \en_p(h)$, with equality if and only if $f=h$.
\end{claim}

\begin{proof}
The collection $\Omega_h$ of functions $f: V \to \R$ such that $f \vert_B = h \vert_{B}$ is convex.  Moreover, the energy functional $\en_p(\cdot)$ is convex on $\Omega_h$. The hypothesis that $h$ is $p$-harmonic yields by \eqref{eq:pharmonic} that $\nabla \en_p(h)=0$ (where the gradient is with respect to values on $I$), and hence $h$ minimizes $\en_p$ on $\Omega_h$.

To prove uniqueness, suppose that $f\in\Omega_h$, $f\ne h$, and $\en_p(f)=\en_p(h)$. Since $G$ is connected and $f=h$ on the nonempty set $B$, there is an edge $\{v,w\}\in E$ for which
\[
    f(v)-f(w)\ne h(v)-h(w).
\]
Strict convexity of $x\mapsto |x|^p$ then gives
\[
    \en_p\!\left(\frac{f+h}{2}\right)
    <\frac{\en_p(f)+\en_p(h)}2
    =\en_p(h),
\]
contradicting minimality of $h$.
    

    

\end{proof}

The $\ell^p$-relaxation dynamics with boundary have a monotonicity property analogous to Claim 1.10 in~\cite{noboundarycase}.

\begin{claim} \label{claim-mono}
Suppose that the opinion profiles $f,g:V \to \R$  satisfy $f \le g$.  If  both profiles are updated at the same interior vertex $v \in I$,
yielding $f_1$ and $g_1$ respectively, then $f_1(v) \le g_1(v)$. 
\end{claim}
\begin{proof}
For a profile $a:V\to\R$, define
\[
    F_a(y):=\sum_{w\sim v}U(y-a(w)).
\]
The function $F_a$ is strictly increasing and has a unique zero, namely the
updated value at $v$.  Since $f(w)\le g(w)$ for every $w\sim v$ and $U$ is
increasing, $F_f(y)\ge F_g(y)$ for every $y\in\R$.  Evaluating at the zero of
$F_g$ shows that the zero of $F_f$ is no larger than the zero of $F_g$.
\end{proof}

\begin{claim} \label{claim: psuperharmonic keep}
    Suppose that $f: V \to \R$ is $p$-superharmonic in $I$ (respectively, $p$-subharmonic in $I$). Let $f_1$ be the function after updating the value at a chosen interior vertex $v \in I$. Then $f_1$ is also $p$-superharmonic in $I$ (respectively, $p$-subharmonic in $I$). 
\end{claim}
\begin{proof}
    We can assume that $f$ is $p$-superharmonic in $I$, as the other case follows by switching $f$ to $-f$. Since $f_1(v)$ is updated, $f_1$ is $p$-harmonic at $v$. Since $f$ is $p$-superharmonic at $v$, we also have $f(v) \ge f_1(v)$. For any interior vertex $v' \neq v$, since $f(v) \ge f_1(v)$ and $f(v') = f_1(v')$, we have
    \begin{align*}
        &p \sum_{w \sim v'} U(f_1(v') - f_1(w)) \ge p \sum_{w \sim v'} U(f(v') - f(w)) \ge 0.
    \end{align*}

    Hence, $f_1$ is $p$-superharmonic at $v'$. 
\end{proof}

The following basic property of $p$-superharmonicity is used in the derivation of upper bounds for expected $\epsilon$-approximation time when $p \in (1, 2)$. 
\begin{claim} \label{claim: monotone path}
    Suppose $f: V \to \R$ is $p$-superharmonic in $I$. Then for any vertex $u \in I$, there exists a simple path $v_0 \sim v_1 \sim \cdots \sim v_\ell = u$ in $G$, such that $v_0 \in B$ and $f(v_0) \le f(v_1) \le \cdots \le f(v_\ell)$. 
\end{claim}
\begin{proof}
    Let $Z \subseteq V$ be the set of vertices $z \in V$ such that there exists a simple path $v_0 \sim v_1 \sim \cdots \sim v_\ell = z$ in $G$ such that $v_0 \in B$ and $f(v_0) \le f(v_1) \le \cdots \le f(v_\ell)$. Suppose that $Z \neq V$. Let   \[ T := \{v \in V \setminus Z: f(v) = \min_{w \in V \setminus Z} f(w)\} \, \]

    and let $v \in T$ be arbitrary.  Since $B \subseteq Z$, we get $v \in I$. Since $v \notin Z$, we infer that for all $z \in Z$ such that $z \sim v$, we have $f(z) > f(v)$.
    
    Therefore, for all $w \sim v$, we have $ f(w) \ge f(v)$. Since $v \in I$ and $f$ is $p$-superharmonic in $I$, \eqref{eq:psuper} holds at $v$. This implies that $f(w) = f(v)$ holds for all $w \sim v$, so $v$ is not adjacent to any vertex in $V \setminus T$. Therefore, $T$ is not connected to any vertex in $V \setminus T$. Since $G$ is connected, this is a contradiction, which refutes our assumption that $Z \neq V$ and completes the proof.
\end{proof}

\subsection{Proof of convergence by Perron's method} \label{sub: perron method}
In this section, we  first  prove    convergence of $f_t$ when $f_0$ is $p$-superharmonic. Then we infer  Fact~\ref{fa:conv to cons} and reduce Theorem~\ref{th:main lpbound} to the special case in which $f_0$ is $p$-superharmonic in $I$. These arguments are standard, but we include them   for convenience.  
\begin{claim}\label{claim: existence and uniqueness of harmonic extension}
    Fix $p>1$. Let $f_0: V \to \R$ be a function that is $p$-superharmonic in $I$. Let $\{v_t\}_{t \ge 1}$ be a sequence of interior vertices such that every vertex $v\in I$ appears in this sequence infinitely often. For each integer $t \ge 1$, let $f_t$ be the function obtained from $f_{t - 1}$ by updating at $v_t$ as in \eqref{eq:pdef}. Then $f_t \ge f_{t+1}$ for all $t \ge 0$, and the pointwise limit $h :=\lim_{t \to \infty} f_t$ is the unique function on $V$ that is $p$-harmonic in $I$ and $h \vert_B = f_0 \vert_B$ .
\end{claim}
\begin{proof}
    By Claim~\ref{claim: psuperharmonic keep} and induction by $t$, we obtain that $f_t$ is $p$-superharmonic in $I$. Since $R_p f_{t - 1} \ge 0$, and the operation at time $t \in \Nstar$ lets $f_t(v_t) := f_{t - 1}(v_t) - R_pf_{t - 1}(v_t)$, we know that $\{f_t\}_{t \ge 0}$ is non-increasing. Every updated value lies between the minimum and maximum of its neighbors, so all profiles remain in the range of $f_0$. In particular, $(f_t)$ is bounded, and we can define
    \begin{align}
        h(v):=\lim_{t\to\infty} f_t(v)\quad \forall v \in V.
    \end{align}
    
    Next, we will verify that $h$ is $p$-harmonic in $I$. Since the set $T_v:=\{t \ge 1: v_t=v\}$ is infinite for any $v \in I$, we have 
    \begin{align}
       \forall t \in T_v, \quad \sum_{w\sim v}U(f_t(v)-f_t(w))=0.
    \end{align}
    Taking the limit as $t \to \infty$ along $T_v$, we obtain
    \begin{align}
        \sum_{w\sim v}U(h(v)-h(w))=0 ,
    \end{align}
    
    whence $h$ is $p$-harmonic in $I$.
    
    It remains to show uniqueness. If $h_1$ and $h_2$ are $p$-harmonic in $I$ and have the same boundary values, Claim~\ref{energy minimizer}, applied with $h=h_1$ and $f=h_2$, shows that equality in the energy comparison can occur only when $h_1=h_2$.
    

\end{proof}

The above claim also implies the following simple fact.
\begin{claim} \label{claim: compare}
    Suppose $f: V \to \R$ is $p$-superharmonic in $I$ and $h: V \to \R$ is a $p$-harmonic function such that $h \vert_{B} = f \vert_{B}$. Then $f(v) \ge h(v)$ for all $v \in V$. 
\end{claim}
\begin{proof}
Run the dynamics from $f$ along any update sequence that visits every interior
vertex infinitely often.  Claim~\ref{claim: existence and uniqueness of harmonic extension}
shows that the profiles decrease pointwise to the $p$-harmonic extension of
$f\vert_B$, which is $h$ by uniqueness.  Hence $f\ge h$.
\end{proof}

Given a function $f_0: B \to \R$, we can extend the domain of $f_0$ to $V$ by letting $f_0(v) := \sup_{w \in B} f_0(w)$ for all $v \in I$. It is not hard to see that $f_0$ is $p$-superharmonic in $I$. By Claim~\ref{claim: existence and uniqueness of harmonic extension}, there exists a unique function $h: V \to \R$ such that $h$ is $p$-harmonic in $I$ and $h \vert_{B} = f_0 \vert_{B}$. Such a function $h: V \to \R$ is called the $p$-{\bf harmonic extension} of $f_0 \vert_{B}$.  

Next, we  sketch the proof of the convergence of the $\ell^p$-relaxation dynamics and confirm  Fact~\ref{fa:conv to cons}. Without loss of generality, we can assume that $f_0(v) \in [0, 1]$ for all $v \in V$. 

Let $v_t$ be the random vertex we   update at time $t$. 
Almost surely, $(v_t)$ visits every vertex in $I$ infinitely often.
Define $f^-_{0}, f^+_{0}$ on $V$ by
\[f^-_{0}(v) = f_0(v) \cdot {\mathbf 1}_{v \in B} \quad \text{and} \quad f^+_{0}(v) = {\bf 1}_{v \in I} +f_0(v) \cdot {\bf 1}_{v \in B}  \,.\]



Then, for each integer $t \ge 1$,  define $f^{-}_t$ to be the function obtained from $f^{-}_{t - 1}$ by updating the value at  $v_t$. We also define $f^{+}_t$ in the same way. 
\begin{prop} \label{prop: transfer to superharmonic case}
    Let $\epsilon \in (0, 1)$ and $f_0: V \to [0, 1]$. Consider
    \begin{gather*}
        \tau_{p}^-(\epsilon) := \inf \{t \in \N: f^{-}_t \ge h - \epsilon \}, \\ \tau_{p}^+(\epsilon) := \inf \{t \in \N: f^{+}_t \le h + \epsilon \} .
    \end{gather*}
    
    Then we have
    \begin{align} \label{eq: transfer to superharmonic case}
        \E[\tau_p(\epsilon)] \le  \E[\tau_p^+(\epsilon)] + \E[\tau_p^-(\epsilon)] \,.
    \end{align}
\end{prop}

\begin{proof}
By Claim~\ref{claim-mono} and induction on $t$, we must have $f_t^{-} \le f_t \le f_t^{+}$. By applying Claim~\ref{claim: existence and uniqueness of harmonic extension} to $f_0^+$ and $1 - f_0^-$, we obtain that both $\{f_t^-\}$ and $\{f_t^+\}$ converge to the $p$-harmonic extension $h$ of $f_0 \vert_B$. Since $f_t^- \le h \le f_t^+$ for all $t \in \N$, we conclude that $\lim_{t \to \infty} f_t = h$ almost surely. 

Moreover, $\tau_{p}(\epsilon) \le \max\left(\tau_{p}^-(\epsilon), \tau_{p}^+(\epsilon)\right)$, since $f_t^- \le f_t \le f_t^+$ and $f_t^-$ is non-decreasing, $f_t^+$ is non-increasing. This implies \eqref{eq: transfer to superharmonic case}.
\end{proof}

\begin{rem} \label{transfer to superharmonic}
    Note that $f_0^+$ and $1 - f_0^{-}$ are both $p$-superharmonic on $I$. Thus, Proposition~\ref{prop: transfer to superharmonic case}  reduces the proof of Theorem~\ref{th:main lpbound} to the special case where ${f_0}$ is $p$-superharmonic in $I$.  
\end{rem}

Let $\mathcal T$ denote the simultaneous update map in \eqref{eq:synch}. Recall that, for $v\in I$,
\[
L_f^v(y)=\sum_{w\sim v}\lvert y-f(w)\rvert^p,
\]
and that $(\mathcal T f)(v)$ is the unique minimizer of $L_f^v$; for $v\in B$, set $(\mathcal T f)(v)=f(v)$.

\begin{prop}\label{prop: synchronous superharmonic reduction}
The map $\mathcal T$ is monotone: if $f\le g$, then $\mathcal T f\le\mathcal T g$.  If $f$ is $p$-superharmonic in $I$, then $\mathcal T f\le f$ and $\mathcal T f$ is again $p$-superharmonic; the analogous statements hold for $p$-subharmonic functions.  Consequently, to prove an upper bound for $\tau_p^{\rm sync}(\epsilon)$ for arbitrary initial profiles in $[0,1]^V$, it suffices to prove the same bound for $p$-superharmonic initial profiles.
\end{prop}
\begin{proof}
Monotonicity follows coordinatewise from Claim~\ref{claim-mono}.  A profile $f$ is $p$-superharmonic exactly when $\mathcal T f\le f$.  Hence, if $f$ is $p$-superharmonic, monotonicity gives
\[
    \mathcal T^2f\le\mathcal T f,
\]
so $\mathcal T f$ is again $p$-superharmonic.  The subharmonic statement follows by replacing $f$ with $-f$.

For the reduction, given $f_0:V\to[0,1]$, define $f_0^+$ to equal $1$ on $I$ and $f_0$ on $B$, and define $f_0^-$ to equal $0$ on $I$ and $f_0$ on $B$.  Then $f_0^-$ is subharmonic, $f_0^+$ is superharmonic, and monotonicity yields
\[
   \mathcal T^t f_0^-\le \mathcal T^t f_0\le \mathcal T^t f_0^+\qquad(t\ge0).
\]
The outer two sequences are monotone and bounded; their pointwise limits are fixed points of the continuous map $\mathcal T$, hence are $p$-harmonic with boundary values $f_0\vert_B$.  By uniqueness they both converge to the same extension $h$.  Thus the middle sequence converges to $h$ as well.  Finally, $1-f_0^-$ is superharmonic and the dynamics commute with the reflection $f\mapsto1-f$, so a bound for superharmonic profiles controls both envelopes and therefore the arbitrary initial profile.
\end{proof}

\section{\texorpdfstring{Convergence rates for $1<p<\infty$ with boundary}{Convergence rates for 1<p<infinity with boundary}}\label{s:lpbounds}

Our main objective in this section is to prove parts {\bf (a)} and {\bf (b)} of Theorem~\ref{th:main lpbound} on the upper bound for expected $\epsilon$-approximation time of the $\ell^p$\textit{-relaxation dynamics}. 

Recall the function $F_*(n,p,D)=n^{-\beta_p}(D/n)^{-\theta_p}$ defined in \eqref{defF}.

\subsection{Lower bounds for expected relative energy decrease or norm decrease}\label{subsec}

Suppose a non-constant $f:V \to [0, 1]$ is $p$-superharmonic in $I$, and define $h:V \to [0, 1]$ to be the $p$-harmonic extension with respect to $f \vert_{B}$. Let $f^\sharp$ be the function obtained from $f$ by updating the value at some interior vertex selected uniformly. 

For all $\alpha>0$ and any function $g: V \to \R$, we define  
 \begin{align} \label{alphanorm} \lVert g \rVert_{\alpha} :=\bigl( \sum_{v \in V} \deg(v) \lvert g(v) \rvert^{\alpha}\bigr)^{\frac{1}{\alpha}} . 
\end{align}

For $\alpha \ge 1$, this is just the weighted $\ell^\alpha$-norm of function $g$. For $\alpha \in (0, 1)$, we call $\lVert g \rVert_\alpha$ the weighted $\ell^\alpha$ quasi-norm of $g$, since it does not satisfy the triangle inequality. 

As noted in the introduction, we have $\en_p(f^{\sharp}) \le \en_p(f)$. Moreover, since $f$ is $p$-superharmonic in $I$, we have $f - h \ge f^{\sharp} - h \ge 0$ by Claim~\ref{claim: compare}. 
The next Proposition gives a lower bound   for the expected  decrease of $f-h$ when $f$ is updated at a random vertex  and $p \in (1, 2)$.   The proof  will be given in Section ~\ref{sec: estimate of norm decrease}. 

\begin{prop} \label{prop:main p<2}
    Fix $p \in (1, 2)$. There exists a constant $c_p > 0$ such that for every $n \ge 2$, any connected graph $G = (V, E)$ with $\lvert V \rvert = n$ , any decomposition $V = I \bigsqcup B$ such that $I, B \neq \varnothing$, the following statement holds: Let $f: V \to [0, 1]$ be any function on $V$ that is $p$-superharmonic in $I$, $h: V \to [0, 1]$ be the $p$-harmonic extension with respect to $f \vert_B$, and $f^\sharp$ be the function after updating the value at a uniformly chosen random vertex in $I$. Then we have
    \begin{align}
        \E\Big[\lVert f^{\sharp} - h \rVert_{p - 1} \Big] \nonumber \le \lVert f - h \rVert_{p - 1} - c_pn^{-2} \lVert f - h \rVert_\infty^{\frac{1}{p - 1}} .
    \end{align}
\end{prop}

\begin{prop}\label{prop:main p>=2}
    Fix $p \in [2, \infty)$. There exists a constant $c_p > 0$ such that for every $n \ge 2$, any connected graph $G = (V, E)$ with $\lvert V \rvert = n$ and average degree $D$, any decomposition $V = I \bigsqcup B$ such that $I, B \neq \varnothing$, the following statement holds: Let $f: V \to \R$ be any function on $V$ that is $p$-superharmonic in $I$, $h: V \to \R$ be the $p$-harmonic extension with respect to $f \vert_B$, and $f^\sharp$ be the function after updating the value at a uniformly chosen random vertex in $I$. Then we have
    \begin{align}
        \E[\en_p(f)-\en_p(f^\sharp)] \ge c_p F_*(n,p,D) \left(\en_p(f) - \en_p(h) \right) .
    \end{align}
\end{prop}

\subsection{Useful inequalities on graphs}\label{s:useful inequalities}

In order to prove Propositions~\ref{prop:main p<2} and~\ref{prop:main p>=2}, we need some auxiliary inequalities on graphs. 

Claims~\ref{claim: elementary ineq 1},~\ref{claim: elementary ineq 2},~\ref{claim: elementary ineq 3} below are elementary and technical inequalities. Claim~\ref{claim: elementary ineq 1} is needed for $p \in (1, 2)$ and Claims~\ref{claim: elementary ineq 2},~\ref{claim: elementary ineq 3} are needed for $p \in [2, \infty)$. 

Throughout this whole article, we always make the convention that $\frac{a}{0} = \infty$ for $a > 0$, $\frac{0}{0} = 1$ and $0^0 = 1$. Recall that in \eqref{eq:defU} we have defined $U(x) := \sign(x) \lvert x \rvert^{p - 1}$.  

\begin{claim} \label{claim: elementary ineq 1}
    Let $p \in (1, 2)$. There exist two constants $c_p, C_p > 0$ such that for all $x,y \in \R$, we have
    \begin{align}
        c_p \lvert y - x \rvert^{p - 1} \min\Big(\frac{\lvert y - x \rvert}{\lvert x \rvert}, 1\Big)^{2 - p} \nonumber &\le \left\lvert U(y) - U(x) \right\rvert \nonumber \\ &\le  C_p\lvert y - x \rvert^{p - 1} \min\Big(\frac{\lvert y - x \rvert}{\lvert x \rvert}, 1\Big)^{2 - p} .
    \end{align}
    By symmetry of $|U(y)-U(x)|$ in $x$ and $y$, the same two-sided estimate also holds with $|x|$ in the denominator replaced by $|y|$.

\end{claim}
\begin{proof}
    If $x = 0$, then we have $\lvert U(y) - U(x) \rvert = \lvert y - x \rvert^{p - 1} \min(\frac{\lvert y - x \rvert}{\lvert x \rvert}, 1)^{2 - p} = \lvert y \rvert^{p - 1}$. Hence we can assume that $x \neq 0$. 
    
    Without loss of generality, we assume that $x = 1$. Otherwise, we can replace $(x, y)$ by $(1, \frac{y}{x})$. We define $F: \R \setminus \{1 \} \to (0,\infty) $ by \[ F(y) := \frac{\lvert U(y) - 1 \rvert}{\lvert y - 1 \rvert^{p - 1} \min(\lvert y - 1 \rvert, 1)^{2 - p}} .\]
    
    Then $F$ is continuous on $\R \setminus \{1 \}$ and \[\lim_{y \to 1} F(y) = \lim_{y \to 1} \frac{\lvert U(y) - 1 \rvert}{\lvert y - 1 \rvert} = U'(1) = p - 1 ,\] so we can extend   $F$  to a continuous function on $\R$. Since   $\lim_{\lvert y \rvert \to \infty} F(y)  = 1$, we can deduce that $F$ has   strictly positive upper and lower bounds  on $\R$.  This implies the existence of $c_p$ and $C_p$. 
   \end{proof}

\begin{claim} \label{claim: elementary ineq 2}
    Let $p \in [2, \infty)$. There exist two constants $c_p, C_p > 0$ such that for any $x, y \in \R$, we have
    \begin{align}
        c_p \lvert y - x \rvert(\lvert y \rvert^{p - 2} + \lvert y - x \rvert^{p - 2}) &\le \lvert U(y) - U(x) \rvert \nonumber \\ &\le C_p \lvert y - x \rvert(\lvert y \rvert^{p - 2} + \lvert y - x \rvert^{p - 2}) .
    \end{align}
\end{claim}
\begin{proof}
    If $x = 0$ then we have $\lvert y - x \rvert(\lvert y \rvert^{p - 2} + \lvert y - x \rvert^{p - 2}) = 2 \lvert y \rvert^{p - 1} = 2 \lvert U(y) \rvert$. Hence, we can assume that $x \neq 0$. 
    
    The following argument is similar to that of Claim~\ref{claim: elementary ineq 1}. Without loss of generality, we also assume that $x = 1$. Otherwise, we can replace $(x, y)$ by $(1, \frac{y}{x})$. 

    We define $F: \R \setminus \{1 \} \to \R$ by \[ F(y) := \frac{\lvert U(y) - 1 \rvert}{\lvert y - 1 \rvert(\lvert y\rvert^{p - 2} + \lvert y - 1 \rvert^{p - 2})}.\]

    Then $F(y)$ is well-defined, positive-valued and continuous on $\R \setminus \{1 \}$. For $p = 2$, we have $F(y) \equiv \frac{1}{2}$. For $p > 2$, we have \[\lim_{y \to 1} F(y) = \lim_{y \to 1} \frac{\lvert U(y) - 1 \rvert}{\lvert y - 1 \rvert} = U'(1) = p - 1 .\]

    So we can always extend $F$ to a continuous function on $\R$. Since \[\lim_{\lvert y \rvert \to \infty} F(y) = \lim_{\lvert y \rvert \to \infty} \frac{\lvert y \rvert^{p - 1}}{2\lvert y \rvert^{p - 1}} = \frac{1}{2} , \]

    we conclude that $F$ has a strictly positive lower bound and a finite upper bound. This implies the existence of $c_p$ and $C_p$. 
\end{proof}

\begin{claim} \label{claim: elementary ineq 3}
    Let $p \in [2, \infty)$. There exist two positive constants $c_p$ and $C_p$ such that
    \begin{align} \label{factoring-square}
        c_p(y - x)^2(\lvert y \rvert^{p - 2} + \lvert y - x \rvert^{p - 2}) &\le \lvert y \rvert^p - \lvert x \rvert^p - p U(x)(y - x) \nonumber \\ &\le C_p(y - x)^2(\lvert y \rvert^{p - 2} + \lvert y - x \rvert^{p - 2}) .
    \end{align}
\end{claim}
\begin{proof}
    The method is also similar to the argument in Claim~\ref{claim: elementary ineq 1} and Claim~\ref{claim: elementary ineq 2}. If $x = 0$, then the expressions in  \eqref{factoring-square} are constant multiples of $\lvert y \rvert^p$, so  we can assume that $x \neq 0$. 
   Without loss of generality, we can also assume that $x = 1$; otherwise, we can replace $(x, y)$ by $(1, \frac{y}{x})$. 

    Define  $F: \R \setminus \{1 \} \to \R$ by \[ F(y) := \frac{\lvert y \rvert^p - 1 - p(y - 1)}{(y - 1)^2(\lvert y \rvert^{p - 2} + \lvert y - 1 \rvert^{p - 2})}.\]

    Since $y \mapsto \lvert y \rvert^p$ is strictly convex and has derivative $p$ at $1$, the numerator in the definition of $F(y)$ is positive when $y \in \R \setminus \{1 \} $. Therefore, $F(y)$ is well-defined, strictly positive, and continuous on $\R \setminus \{1 \}$. For $p = 2$, we have $F(y) \equiv \frac{1}{2}$. For $p > 2$, by Taylor's expansion, we have that as $y \to 1$, \[ y^p  - 1 - p(y - 1)   = \frac{p(p - 1)}{2}(y - 1)^2 + o((y - 1)^2).\] 
    
    This implies that for $p > 2$, we have \[\lim_{y \to 1} F(y) = \lim_{y \to 1} \frac{\lvert y^p  - 1 - p(y - 1)\rvert}{(y - 1)^2} =  \frac{p(p - 1)}{2} .\]

    So we can always extend $F$ to a continuous function on $\R$. Since \[\lim_{\lvert y \rvert \to \infty} F(y) = \lim_{\lvert y \rvert \to \infty} \frac{\lvert y \rvert^p}{2\lvert y \rvert^p} = \frac{1}{2} , \]

    we conclude that $F$ has a strictly positive lower bound and a finite upper bound. This implies the existence of $c_p$ and $C_p$. 
\end{proof}

Lemma~\ref{Lemma: poincare inequality} below is a version of an $\ell^p$-Poincar\'e inequality on graphs. 

\begin{lem} \label{Lemma: poincare inequality}
    Let $p \in [1, \infty)$ and $G = (V, E)$ be a finite connected graph with $n \ge 2$ vertices and average degree $D$. Let $V = I \bigsqcup B$ be a decomposition of the vertex set, where $I, B \neq \varnothing$. Let $f: V \to \R$ be any function on $V$ such that $f \vert_{B} = 0$. Then, 
    \begin{align} \label{eq:poinc}
        \en_p(f) = \sum_{\{v, w \} \in E} \lvert f(v) - f(w) \rvert^p \ge \frac{1}{n^pD} \lVert f\rVert_p^p .
    \end{align}
\end{lem}

\begin{proof}
     For any $v_0 \in V$, since $G$ is connected, we can find a simple path $v_0 \sim v_1 \sim \cdots \sim v_\ell \in B$ where $\ell \le n$. 

    Applying Jensen's inequality to the function $x \rightarrow x^p$ on $[0,\infty)$ gives
    \begin{align*}
        \en_p(f) &= \sum_{e = \{v, w \} \in E} \lvert f(v) - f(w) \rvert^p \\ &\ge \sum_{i = 0}^{\ell - 1} \lvert f(v_i) - f(v_{i + 1}) \rvert^p \\ &\ge \frac{1}{\ell^{p - 1}}\Big(\sum_{i = 0}^{\ell - 1} \lvert f(v_i) - f(v_{i + 1}) \rvert \Big)^p \\ &\ge \frac{1}{n^{p - 1}} \lvert f(v_0) \rvert^p .
    \end{align*}

    Hence, the right-hand side of \eqref{eq:poinc} can be bounded above as follows:
    \begin{align*}
        \frac{1}{n^pD} \lVert f \rVert_p^p &= \frac{1}{n^pD} \sum_{v_0 \in V} \deg(v_0) \lvert f(v_0) \rvert^p \\ &\le \frac{1}{n^pD} \sum_{v_0 \in V} \deg(v_0) {n^{p - 1}} \en_p(f) = \en_p(f) .
    \end{align*}
\end{proof}

Lemma~\ref{Lemma: essential ineq for p<2} below is more specialized;  it is used in the proof of Proposition~\ref{prop:main p<2}. 

\begin{lem} \label{Lemma: essential ineq for p<2}
    Let $p \in (1, 2)$ and $G = (V, E)$ be a finite connected graph with $n \ge 2$ vertices. Let $V = I \bigsqcup B$ be a decomposition of the vertex set, where $I, B \neq \varnothing$. Let $f, g: V \to [0, 1]$ be any two functions on $V$. Suppose $f$ is $p$-superharmonic in $I$ and $g \vert_{B} = 0$. Then we have
    \begin{align}
        \sum_{\{v, w\} \in E} \lvert g(v) - g(w) \rvert^p \min\Big(\frac{\lvert g(v) - g(w)\rvert}{\lvert f(v) - f(w) \rvert}, 1\Big)^{2 - p} \ge \frac{1}{2n^{p - 1}} \lVert g \rVert_\infty^2 .
    \end{align}
\end{lem}
\begin{proof}
    The case $g \equiv 0$ is trivial. Without loss of generality, we can assume that $g(v) = \lVert g \rVert_\infty > 0$ for some $v \in I$.  
    
    By Claim~\ref{claim: monotone path}, there exists a simple path $v_0 \sim v_1 \sim \cdots \sim v_{\ell} = v$ such that $v_0 \in B$ and $f(v_0) \le f(v_1) \le \cdots \le f(v_{\ell})$. Since $v_0, \cdots, v_\ell$ are distinct, we have $\ell \le n$. We also have 
    \begin{align} \label{transfer}
        &\sum_{\{v, w\} \in E} \lvert g(v) - g(w) \rvert^p \min\Big(\frac{\lvert g(v) - g(w)\rvert}{\lvert f(v) - f(w) \rvert}, 1\Big)^{2 - p} \nonumber \\ &\ge \sum_{i = 1}^{\ell} \lvert g(v_i) - g(v_{i - 1}) \rvert^p \min\Big(\frac{\lvert g(v_i) - g(v_{i - 1}) \rvert}{\lvert f(v_i) - f(v_{i - 1}) \rvert}, 1\Big)^{2 - p} .
    \end{align}

    We define
    \begin{gather*}
        I_1 := \left\{i \in [1, \ell] \cap \Z : \lvert g(v_i) - g(v_{i - 1}) \rvert > 0, \frac{\lvert g(v_i) - g(v_{i - 1}) \rvert}{\lvert f(v_i) - f(v_{i - 1}) \rvert} \le 1\right\} , \\ I_2 := \left\{i \in [1, \ell] \cap \Z : \frac{\lvert g(v_i) - g(v_{i - 1}) \rvert}{\lvert f(v_i) - f(v_{i - 1}) \rvert} > 1\right\} .
    \end{gather*}

    We also define \[ a := \min\left(\sum_{i \in I_1} \lvert g(v_i) - g(v_{i - 1}) \rvert, 1\right), \quad  b := \min\left(\sum_{i \in I_2} \lvert g(v_i) - g(v_{i - 1}) \rvert, 1\right) .\]

    Since $\sum_{i = 1}^{\ell} \lvert g(v_i) - g(v_{i - 1}) \rvert \ge g(v_\ell) = \lVert g \rVert_\infty$, and $\lVert g \rVert_\infty \le 1$, we get that 
    \begin{align}\label{le3.7eqn1}
        a + b \ge \lVert g \rVert_\infty.
    \end{align}
    
    The right-hand side of \eqref{transfer} can be estimated by:
    \begin{align*}
        &\sum_{\{v, w\} \in E} \lvert g(v) - g(w) \rvert^p \min\Big(\frac{\lvert g(v) - g(w)\rvert}{\lvert f(v) - f(w) \rvert}, 1\Big)^{2 - p} \\ \ge &\sum_{i \in I_1} \frac{\lvert g(v_i) - g(v_{i - 1}) \rvert^2}{\lvert f(v_i) - f(v_{i - 1}) \rvert^{2-p}} + \sum_{i \in I_2} \lvert g(v_i) - g(v_{i - 1}) \rvert^p .
    \end{align*}

    In order to estimate the contribution for $\sum_{i \in I_1}$ for $I_1 \neq \varnothing$, we use the Cauchy-Schwarz inequality as follows:
    \begin{align} \label{cauchy in the proof of p<2 ineq}
        \sum_{i \in I_1} \frac{\lvert g(v_i) - g(v_{i - 1}) \rvert^2}{\lvert f(v_i) - f(v_{i - 1})\rvert^{2 - p}} & \ge \frac{1}{\sum_{i \in I_1} \lvert f(v_i) - f(v_{i - 1}) \rvert^{2 - p}} \left( \sum_{i \in I_1} \lvert g(v_i) - g(v_{i - 1}) \rvert\right)^2 \nonumber \\ &\ge \frac{a^2}{\sum_{i \in I_1} \lvert f(v_i) - f(v_{i - 1}) \rvert^{2 - p}} \quad \text{for } I_1 \neq \varnothing .
    \end{align}

    Since $0\le f\le1$ and $f(v_i)$ is non-decreasing along the chosen path $v_0,\ldots,v_\ell$, we have
    \[
        \sum_{i\in I_1}|f(v_i)-f(v_{i-1})|
        \le \sum_{i=1}^{\ell}(f(v_i)-f(v_{i-1}))
        =f(v_\ell)-f(v_0)\le1.
    \]
    (The same inequality is trivial when $I_1=\varnothing$.)  Therefore the power mean inequality gives
    \begin{align} \label{simple bound of the denominator}
        \sum_{i \in I_1} \lvert f(v_i) - f(v_{i - 1}) \rvert^{2 - p} \le \lvert I_1 \rvert ^{p - 1} \left(\sum_{i \in I_1} \lvert f(v_i) - f(v_{i - 1}) \rvert\right)^{2 - p} \le n^{p - 1} .
    \end{align}
    
    Hence by \eqref{cauchy in the proof of p<2 ineq} and \eqref{simple bound of the denominator}, we have
    \begin{align} \label{le3.7eqn2} 
        \sum_{i \in I_1} \frac{\lvert g(v_i) - g(v_{i - 1}) \rvert^2}{\lvert f(v_i) - f(v_{i - 1})\rvert^{2 - p}} \ge \frac{a^2}{n^{p - 1}} .
    \end{align}

    Again using power mean inequality, we can estimate the contribution $\sum_{i \in I_2}$ by
    \begin{align}\label{le3.7eqn3}
        \sum_{i \in I_2} \lvert g(v_i) - g(v_{i - 1}) \rvert^p &\ge \frac{1}{\lvert I_2 \rvert^{p - 1}}\left(\sum_{i \in I_2} \lvert g(v_i) - g(v_{i - 1}) \rvert \right)^p \nonumber \\ &\ge \frac{b^p}{n^{p - 1}} \ge \frac{b^2}{n^{p - 1}} .
    \end{align}
    (The same inequality is trivial when $I_2=\varnothing$.)  
    
    Combining \eqref{transfer}, \eqref{le3.7eqn1}, \eqref{le3.7eqn2} and \eqref{le3.7eqn3}, we deduce that:
    \begin{align*}
        &\sum_{\{v, w\} \in E} \lvert g(v) - g(w) \rvert^p \min\Big(\frac{\lvert g(v) - g(w)\rvert}{\lvert f(v) - f(w) \rvert}, 1\Big)^{2 - p} \\ & \ge \sum_{i \in I_1} \frac{\lvert g(v_i) - g(v_{i - 1}) \rvert^2}{\lvert f(v_i) - f(v_{i - 1}) \rvert^{2-p}} + \sum_{i \in I_2} \lvert g(v_i) - g(v_{i - 1}) \rvert^p \nonumber \\ & \ge \frac{a^2}{n^{p - 1}} + \frac{b^2}{n^{p - 1}} \ge \frac{1}{2n^{p - 1}} \lVert g\rVert_\infty^2 .
    \end{align*}
\end{proof}

We will use the following lemma in the proof of Proposition~\ref{prop:main p>=2}. 

\begin{lem} \label{lem: spectral gap}
    Let $n \ge 2$ and let $G = (V, E)$ be a finite connected graph with $V = \{ 1, 2, \cdots, n \}$ and average degree $D$. Let $V = I \bigsqcup B$ be the decomposition of $V$ such that $I, B \neq \varnothing$. Fix $\alpha \in [0, 1]$. Let $(x_i)_{1 \le i \le n}$ be a sequence of non-negative real numbers, and let $(a_{i, j})_{1 \le i < j \le n}$ be an array with non-negative elements. Suppose the following conditions are satisfied:
    \smallskip  
        \begin{itemize}
        \item[{\bf(a)}] $a_{i, j} = 0$ when $\{i, j \} \notin E$.  
    \smallskip    
        \item[{\bf(b)}] $x_i = 0$ when $i \in B$. 
    \smallskip
        \item[{\bf(c)}] $\sum_{j = 1}^{i - 1} a_{j, i} \ge \sum_{j = i + 1}^n a_{i, j}$ when $i \in I$. 
    \end{itemize}
    \medskip

    Then we have
    \begin{align} \label{eq:lambound}
        \sum_{\substack{1 \le i < j \le n \\ i \sim j}} a_{i, j}^{\alpha}(x_j - x_i)^2 \ge \lambda(n, \alpha, D) \sum_{\substack{1 \le i < j \le n \\ i \sim j}} a_{i, j}^{\alpha}(x_i^2 + x_j^2) ,
    \end{align}
    
    where for some constants $c_\alpha>0$, 
    \begin{align*}
        \lambda(n,\alpha,D):=
        \begin{cases}
            c_{\alpha}n^{-2}D^{2 \alpha - 1} \, \ \, & \alpha \in [0, \frac{1}{2}) ,\\
            c_{\alpha}n^{-2} (1 + \log D)^{-1/2} \, \ \, & \alpha = \frac{1}{2} ,\\ c_{\alpha}n^{-2}
             \, \ \, &\alpha \in (\frac{1}{2}, 1] .
        \end{cases}
    \end{align*}

    Furthermore, if we assume the following extra condition
    \smallskip  
        \begin{itemize}
        \item[{\bf(d)}] For any $i < j < k$ with $i \sim j, i \sim k$, we have $a_{i, j} \le a_{i, k}$, 
        \end{itemize}
    \medskip
    
    then we have
    \begin{align} \label{eq: lambound2}
        \sum_{\substack{1 \le i < j \le n \\ i \sim j}} a_{i, j}^\alpha (x_j - x_i)^2 \ge \lambda_*(n, \alpha, D) \sum_{\substack{1 \le i < j \le n \\ i \sim j}} a_{i, j}^{\alpha}(x_i^2 + x_j^2) , 
    \end{align}

    where $\lambda_*(n, \alpha, D)$ is defined by 
        \begin{align*}
        \lambda_*(n,\alpha,D):=
        \begin{cases}
            c_{\alpha}n^{-2}D^{2 \alpha - 1} \, \ \, & \alpha \in [0, \frac{1}{2}) , \\ c_{\alpha}n^{-2}
             \, \ \, &\alpha \in [\frac{1}{2}, 1] .
        \end{cases}
    \end{align*}
\end{lem}

In order to prove Lemma~\ref{lem: spectral gap}, we need the following proposition.

\begin{prop} \label{prop: spectral gap and bottleneck ratio}
    Let $G = (V, E)$ be a finite simple weighted graph (with no loops or multiple edges), with symmetric edge weights $w_{x, y} = w_{y, x} \ge 0$ for $x \sim y$. Let $V = I \bigsqcup B$ be any decomposition such that $I, B \neq \varnothing$ and let $m: V \to [0, \infty)$ be a non-negative function on $V$ with $m(I) > 0$. 

    For a function $f:V \to \R$, we define the weighted $L^2$ semi-norm with respect to $m$ by \[ \lVert f \rVert_{L^2(m)}^2 := \sum_{v \in V} m(v) f(v)^2 .\]

    Moreover, we define the Dirichlet energy $\en(f, f)$ to be  
    \begin{align}
        \en(f, f) := \frac{1}{2} \sum_{x \sim y} w_{x, y} \big(f(x) - f(y)\big)^2 .
    \end{align}
    
    The first Dirichlet eigenvalue with respect to boundary $B$ is defined by
    \begin{align}
        \lambda_D:= \inf_{\substack{\lVert f \rVert_{L^2(m)} > 0\\ f|_B=0}} \frac{\en(f,f)}{\lVert f \rVert_{L^2(m)}^2} .
    \end{align}
    Define the Dirichlet bottleneck ratio by
    \begin{align}
        \Phi_B := \min_{\substack{\varnothing \neq S \subseteq I \\ m(S) > 0}} \frac{w(S, V \setminus S)}{m(S)} ,
    \end{align}
    
    where \[ w(S,V\setminus S) := \sum_{\substack{x \in S, y \notin S \\ x \sim y}} w_{x, y}, \qquad m(S):=\sum_{x\in S}m(x) . \]
    
    Then we have $\lambda_D \ge \frac{\Phi_B}{\lvert I \rvert}$. 
\end{prop}

\begin{proof}
    We fix the function $f: V \to \R$ with $f \vert_B = 0$ and aim to prove that
    \begin{align}
        \en(f, f) \ge \frac{\Phi_B}{\lvert I \rvert} \cdot \lVert f \rVert_{L^2(m)}^2 .
    \end{align}

    Since $\en(f, f) \ge \en(\lvert f \rvert, \lvert f \rvert)$, we assume that $f \ge 0$. 
    
    We reorder the vertices in $B$ by $v_{0}, v_{-1}, \cdots, v_{-\lvert B \rvert + 1}$ and vertices in $I$ by \\ $v_1, \cdots, v_{\lvert I \rvert}$, where we assume that \[0 \le f(v_1) \le f(v_2) \le \cdots \le f(v_{\lvert I \rvert}) .\]
    
    Then we have
    \begin{align}
        \en(f, f) &= \sum_{\substack{i < j, j \ge 1 \\ v_i \sim v_j}} w_{v_i, v_j}\bigl(f(v_i) - f(v_j)\bigr)^2 \nonumber \\ &\ge \sum_{\substack{i < j, j \ge 1 \\ v_i \sim v_j}} w_{v_i, v_j} \left( \sum_{i < k \le j, k \ge 1} \bigl(f(v_k) - f(v_{k - 1}) \bigr)^2\right) \nonumber \\ &= \sum_{1 \le k \le \lvert I \rvert} \Bigl(\sum_{\substack{i < k \le j \\ v_i \sim v_j}} w_{v_i, v_j}\Bigr) \bigl( f(v_k) - f(v_{k - 1})\bigr)^2 . 
    \end{align}

    By applying $w(S, V \setminus S) \ge \Phi_B m(S)$ to $S = \{v_k, \cdots, v_{\lvert I \rvert} \}$, we have
    \begin{align}
        \en(f, f) &\ge \Phi_B \sum_{1 \le k \le \lvert I \rvert} \Bigl(\sum_{i \ge k} m(v_i) \Bigr) \cdot \bigl( f(v_k) - f(v_{k - 1})\bigr)^2 \nonumber \\ &\ge \Phi_B \sum_{1 \le i \le \lvert I \rvert} m(v_i) \Bigl(\sum_{1 \le k \le i} \bigl(f(v_k) - f(v_{k - 1}) \bigr)^2 \Bigr) \nonumber \\ &\ge \frac{\Phi_B}{\lvert I \rvert} \cdot \sum_{1 \le i \le \lvert I \rvert} m(v_i) f(v_i)^2 = \frac{\Phi_B}{\lvert I \rvert} \cdot \lVert f \rVert_{L^2(m)}^2 .
    \end{align}

    where we use the Cauchy-Schwarz inequality in the third step. 

    This proves that $\lambda_D \ge \frac{\Phi_B}{\lvert I \rvert}$. 
\end{proof}

We now return to the proof of Lemma~\ref{lem: spectral gap}. 
\begin{proof}[Proof of Lemma~\ref{lem: spectral gap}]
    The $\alpha = 0$ case is proved in Lemma~\ref{Lemma: poincare inequality}. Hence we assume that $\alpha \in (0, 1]$. We assign weight $w_{i, j} = w_{j, i} = a_{i, j}^\alpha$ for $i \sim j$ and define the measure $m: V \to \R$ by $m(i) = \sum_{j \sim i} w_{i, j}$. Both sides equal $0$ when $m(I) = 0$, hence we can assume that $m(I) > 0$. 

    The lemma reduces to the lower bound estimate for the first Dirichlet eigenvalue $\lambda_D$. Thanks to Proposition~\ref{prop: spectral gap and bottleneck ratio}, it suffices to estimate the Dirichlet bottleneck ratio $\Phi_B$. Fix $\varnothing \neq S \subseteq I$. We define $E_S, \partial_S E \subseteq E$ by
    \begin{align}
        E_S &:= \{\{i, j \} \in E: i < j, i, j \in S \} , \\ \partial_S E &:= \{ \{i, j \} \in E: i < j, \text{exactly one of }i \text{ or } j \text{ belongs to } S\}. 
    \end{align}

    We can assume that $E_S \neq \varnothing$, otherwise we have $w(S, V \setminus S) = m(S)$. For each vertex $i \in V$, we define $R_i^+, R_i^-$ by
    \begin{align*}
        R_i^+ &:= \sum_{\substack{j < i \\ j \sim i}} a_{j, i} \ , \quad R_i^- := \sum_{\substack{j > i \\ j \sim i}} a_{i, j} . 
    \end{align*}
    
    By condition {\bf (c)}, we have $R_i^+ \ge R_i^-$ for $i \in I$. 

    Since $\alpha \in (0, 1]$, we have
    \begin{align} \label{eq: eq1 for bottleneck}
        w(S, V \setminus S) = \sum_{\substack{i < j \\ \{ i, j \} \in \partial_S E}} w_{i, j} \ge \Bigg( \sum_{\substack{i < j \\ \{ i, j \} \in \partial_S E}} a_{i, j}\Bigg)^\alpha .
    \end{align}

    Since $R_i^+ \ge R_i^-$, we have 
    \begin{align} \label{eq: eq2 for bottleneck}
        \sum_{\substack{i < j \\ \{ i, j \} \in \partial_S E}} a_{i, j} &\ge \sum_{\substack{i < j \\ \{i, j \} \in \partial_S E}} a_{i, j} - \sum_{i \in S} \frac{n + 1 - i}{n} (R_i^+ - R_i^-) \nonumber \\ &\ge \sum_{i < j} \Bigl(1_{\{i, j \} \in \partial_S E} + \frac{n + 1 - i}{n} 1_{i \in S}- \frac{n + 1 - j}{n} 1_{j \in S} \Bigr) a_{i, j} \nonumber \\ &\ge \frac{1}{n} \sum_{\substack{i < j \\ \{ i, j \} \in E_S}} (j - i) a_{i, j} .
    \end{align}

    We first assume that $\alpha < 1$. By \hld's inequality, we have
    \begin{align}
        \Bigg( \sum_{\substack{i < j \\ \{i, j \} \in E_S}} \frac{1}{(j - i)^{\frac{\alpha}{1 - \alpha}}}\Bigg)^{1 - \alpha}\Bigg(\sum_{\substack{i < j \\ \{i, j \} \in E_S}} (j - i)a_{i, j}\Bigg)^\alpha \ge \sum_{\substack{i < j \\ \{ i, j \} \in E_S}} w_{i, j} .
    \end{align}

    For $\alpha < 1$, we set \[ F := \sum_{\substack{i < j \\ \{ i, j \} \in E_S}} \frac{1}{(j - i)^{\frac{\alpha}{1 - \alpha}}} .\]

    Specifically, we set $F = 1$ when $\alpha = 1$. 
    
    Combining \eqref{eq: eq1 for bottleneck} and \eqref{eq: eq2 for bottleneck}, for $\alpha < 1$, we get
    \begin{align}
        w(S, V \setminus S) \ge \frac{1}{n^\alpha F^{1 - \alpha}} \sum_{\substack{i < j \\ \{i, j \} \in E_S}} w_{i, j} .
    \end{align}

    Note that the preceding inequality also holds for $\alpha = 1$, which directly follows from \eqref{eq: eq2 for bottleneck}. 
    
    Since \[m(S) = w(S, V \setminus S) + 2 \sum_{\substack{i < j \\ \{i, j \} \in E_S}} w_{i, j} , \]

    we have $\Phi_B \ge \frac{1}{2n^\alpha F^{1 - \alpha} + 1}$. We estimate $F$ under three cases: $\alpha \in (0, \frac{1}{2})$, $\alpha = \frac{1}{2}$ and $\alpha \in (\frac{1}{2}, 1]$. 

    {\bf Case 1: $\alpha \in (0, \frac{1}{2})$: }
    
        In this case, there exists a constant $C_\alpha > 0$ such that
        \begin{align*}
            F &\leq nD \frac{1}{D^{\frac{\alpha}{1 - \alpha}}} + \sum_{\substack{1 \le i < j \le n \\ i \sim j}} \left( \frac{1}{(j - i)^{\frac{\alpha}{1 - \alpha}}} - \frac{1}{D^{\frac{\alpha}{1 - \alpha}}}\right) \\ & \le nD^{\frac{1 - 2\alpha}{1 - \alpha}} + \sum_{\substack{1 \le i < j \le n \\ j - i \le D}} \frac{1}{(j - i)^{\frac{\alpha}{1 - \alpha}}} \\ & \le nD^{\frac{1 - 2\alpha}{1 - \alpha}} + n \sum_{1 \le i \le D} \frac{1}{i^{\frac{\alpha}{1 - \alpha}}} \le C_\alpha nD^{\frac{1 - 2\alpha}{1 - \alpha}} .
        \end{align*}

    {\bf Case 2: $\alpha = \frac{1}{2}$: }

        In this case, there exists a constant $C_\alpha > 0$ such that
        \begin{align*}
            F &\le \sum_{\substack{i < j \\ i \sim j}} \frac{1}{j - i} \nonumber \\ &\le \frac{\lvert E \rvert}{D} + \sum_{\substack{i < j \\ i \sim j, j - i \le D}} \Bigl(\frac{1}{j - i} - \frac{1}{D}\Bigr) \nonumber \\ &\le n + n\sum_{1 \le i \le D} \frac{1}{i} \le C_\alpha n(\log D + 1) . 
        \end{align*}

    {\bf Case 3: $\alpha \in (\frac{1}{2}, 1)$: }
        
        In this case, there exists a constant $C_\alpha > 0$ such that \[ F \le \sum_{1 \le i < j \le n} \frac{1}{(j - i)^{\frac{\alpha}{1 - \alpha}}} \le n \sum_{i = 1}^{n} \frac{1}{i^{\frac{\alpha}{1 - \alpha}}} \le C_\alpha n .\]

    When $\alpha=1$, we use the convention $F=1$ introduced above, and the preceding bottleneck estimate follows directly from \eqref{eq: eq2 for bottleneck}.

    By Proposition~\ref{prop: spectral gap and bottleneck ratio}, for some chosen $c_\alpha > 0$, we have
    \begin{align}
        \lambda_D \ge \frac{\Phi_B}{\lvert I \rvert} \ge \frac{1}{n(2n^\alpha F^{1 - \alpha} + 1)} \ge \lambda(n, \alpha, D) .
    \end{align}

    We now deal with the case in which condition {\bf (d) } is satisfied. Since we already get the tight bound of $\lambda_D$ for $\alpha \in [0, \frac{1}{2})$ and $\alpha = 1$ without condition {\bf (d) }, we can suppose that $\alpha \in [\frac{1}{2}, 1)$. For any $i \in S$, we define \[d_i^+ := \# \{j: j > i, \{i, j \} \in E_S \} .\]

    By condition {\bf (d) } and Chebyshev's sum inequality, we have
    \begin{align} \label{eq: chebyshev's inequality}
        \sum_{\substack{i < j \\ \{i, j \} \in E_S}} (j - i) a_{i, j} &\ge \sum_{\substack{i \in S \\ d_i^+ > 0} }\frac{1}{d_i^+}\Bigg( \sum_{\substack{j > i \\ \{i, j \} \in E_S}} (j - i)\Bigg) \Bigg( \sum_{\substack{j > i \\ \{i ,j \} \in E_S}} a_{i, j}\Bigg) \nonumber \\ &\ge \frac{1}{2} \sum_{\substack{i < j \\ \{i, j \} \in E_S}} d_i^+ a_{i, j} . 
    \end{align}

    By \hld's inequality, we have
    \begin{align} \label{eq: holder inequality2}
        \Bigg( \sum_{\substack{i < j \\ \{i, j \} \in E_S}} \frac{1}{(d_i^+)^{\frac{\alpha}{1 - \alpha}}}\Bigg)^{1 - \alpha} \Bigg( \sum_{\substack{i < j \\ \{i, j \} \in E_S}} d_i^+ a_{i, j} \Bigg)^\alpha \ge \sum_{\substack{i < j \\ \{ i, j \} \in E_S}} w_{i, j} .
    \end{align}

    Since $\alpha \in [\frac{1}{2}, 1)$, we have
    \begin{align}
         \sum_{\substack{i < j \\ \{i, j \} \in E_S}} \frac{1}{(d_i^+)^{\frac{\alpha}{1 - \alpha}}} \le \sum_{\substack{i < j \\ \{i, j \} \in E_S}} \frac{1}{d_i^+} \le n .
    \end{align}

    Combining \eqref{eq: eq1 for bottleneck}, \eqref{eq: eq2 for bottleneck}, \eqref{eq: chebyshev's inequality} and \eqref{eq: holder inequality2}, we get
    \begin{align}
        w(S, V \setminus S) &\ge \Bigg( \frac{1}{2n}\sum_{\substack{i < j \\ \{i, j \} \in E_S}} d_i^+ a_{i, j} \Bigg)^\alpha \nonumber \\ &\ge \frac{1}{2n^\alpha \cdot n^{1 - \alpha}} \sum_{\substack{i < j \\ \{i, j \} \in E_S}} w_{i, j} \nonumber \\ &\ge \frac{1}{2n} \sum_{\substack{i < j \\ \{i, j \} \in E_S}} w_{i, j} .
    \end{align}

    By Proposition~\ref{prop: spectral gap and bottleneck ratio}, we have \[ \lambda_D \ge \frac{1}{n} \Phi_B \ge \frac{1}{5n^2}.\]

    This concludes Lemma~\ref{lem: spectral gap}. 
\end{proof}

\subsection{\texorpdfstring{Estimate of $\ell^{p - 1}$ quasi-norm decrease for $p \in (1, 2)$}{Estimate of ell-(p-1) quasi-norm decrease for 1<p<2}} \label{sec: estimate of norm decrease}

With the help of Lemma~\ref{Lemma: essential ineq for p<2}, we can prove Proposition~\ref{prop:main p<2}. 

\begin{proof}[Proof of Proposition~\ref{prop:main p<2}]
    We define $g := f - h$. Since $f$ is $p$-superharmonic in $I$ and $h$ is $p$-harmonic in $I$, Claim~\ref{claim: compare} gives $f-R_pf\ge h$, and hence $0\le R_pf\le f-h=g$. We can also assume that $\lVert g \rVert_\infty > 0$, since the case $g=0$ is trivial.
    
    Recall the notation $U(x)=\sign(x)\lvert x \rvert^{p-1}$. By  \eqref{eq:pharmonic} applied to $p$-harmonic functions in $I$,  for every $v \in I$, we have
    \begin{align*}
        &\sum_{w: w \sim v} U(f(v) - R_p f(v) - f(w)) = \sum_{w: w \sim v} U\big(h(v) - h(w)\big) = 0 .
    \end{align*}

    This implies that for fixed $v \in I$, we have 
    \begin{align}\label{eq: single vertex}
        &\sum_{w: w \sim v} \Big(U\big(f(v) - f(w)\big)- U(f(v) - R_p f(v) - f(w)) \Bigr)\nonumber \\ 
        = &\sum_{w: w \sim v} \Bigl(U\big(f(v) - f(w)\big) - U\big(h(v) - h(w)\big) \Bigr).
    \end{align}

    Applying Claim~\ref{claim: elementary ineq 1} to the left-hand side of \eqref{eq: single vertex} , we deduce that there exists a constant $C_{1, p} > 0$ such that for every $v \in I$, we have
    \begin{align} \label{ineq: single vertex p<2}
        &C_{1, p} \sum_{w: w \sim v} (R_p f(v))^{p - 1} \min\Big(\frac{R_p f(v)}{\lvert f(v) - f(w) \rvert}, 1\Big)^{2 - p} \nonumber \\ \ge &\sum_{w: w \sim v} \Bigl(U\big(f(v) - f(w)\big) - U\big(h(v) - h(w)\big) \Bigr),
    \end{align}

    Summing up \eqref{ineq: single vertex p<2} for all $v \in I$ with coefficient $g(v)$, we have
    \begin{align} \label{ineq: p<2 global ineq 1}
        &C_{1, p} \sum_{v \in V} g(v) (R_p f(v))^{p - 1} \sum_{w \sim v} \min\Big(\frac{R_p f(v)}{\lvert f(v) - f(w) \rvert}, 1\Big)^{2 - p} \nonumber \\ \ge &\sum_{v \in V} g(v) \sum_{w \sim v} \Bigl(U\big(f(v) - f(w)\big) - U\big(h(v) - h(w)\big) \Bigr) \\
        =& \sum_{\{v, w \} \in E} \big\lvert (f(v) - h(v)) - (f(w) - h(w)) \big\rvert\cdot \nonumber \big \lvert U\big(f(v) - f(w)\big) - U\big(h(v) - h(w)\big) \big\rvert ,
    \end{align}

    where we use the fact that $g \vert_B = 0$. The last equality holds since $U$ is strictly increasing. 

    Again applying Claim~\ref{claim: elementary ineq 1} to the right-hand side of \eqref{ineq: p<2 global ineq 1}, we get that there exists a constant $c_{1, p} > 0$ such that
    \begin{align} \label{ineq: p<2 global ineq 2}
        & C_{1, p} \sum_{v \in V} g(v) (R_p f(v))^{p - 1} \sum_{w \sim v} \min\Big(\frac{R_p f(v)}{\lvert f(v) - f(w) \rvert}, 1\Big)^{2 - p} \nonumber \\ \ge \ & c_{1, p} \sum_{\{v, w\} \in E} \lvert g(v) - g(w) \rvert^p \min\Big(\frac{\lvert g(v) - g(w) \rvert}{\lvert f(v) - f(w) \rvert}, 1\Big)^{2 - p} .
    \end{align}

    By Lemma~\ref{Lemma: essential ineq for p<2}, the right-hand side of \eqref{ineq: p<2 global ineq 2} can be bounded below by $\frac{c_{1, p}}{2n^{p - 1}} \lVert g \rVert_\infty^2$. Moreover, the left-hand side of \eqref{ineq: p<2 global ineq 2} can be bounded above as follows:
    \begin{align*}
        &C_{1, p} \sum_{v \in V} g(v) (R_p f(v))^{p - 1} \sum_{w \sim v} \min\Big(\frac{R_p f(v)}{\lvert f(v) - f(w) \rvert}, 1\Big)^{2 - p} \nonumber \\ \le \ &C_{1, p} \lVert g \rVert_\infty\sum_{v \in V} \deg(v) (R_p f(v))^{p - 1} = C_{1, p} \lVert g \rVert_\infty \lVert R_p f \rVert_{p - 1}^{p - 1} .
    \end{align*}

    This implies that
    \begin{align*}
        \frac{c_{1, p}}{2n^{p - 1}} \lVert g \rVert_\infty^2 \le C_{1, p} \lVert g \rVert_\infty \lVert R_p f \rVert_{p - 1}^{p - 1} .
    \end{align*}
    
    Hence there exists a constant $c_{2, p} \in (0, \frac{1}{2})$ such that
    \begin{align} \label{ineq: norm of delta}
        \lVert R_p f \rVert_{p - 1}^{p - 1} \ge \frac{c_{2, p}}{n^{p - 1}} \lVert g \rVert_\infty .
    \end{align}

    Set $q := p - 1$, $A = \lVert g \rVert_q$ and \[ \delta := c_{2,p}^{1/q}n^{-1}\|g\|_\infty^{1/q} , \]
    
    Then $\lVert R_pf \rVert_q \ge \delta$. Since $q \in (0, 1)$, reverse Minkowski gives \[ \lVert g - R_p f \rVert_q \le \lVert g \rVert_q - \lVert R_p f \rVert_q \le A - \delta .\]

    Since $x \mapsto x^q$ is concave, we have
    \[ (A - \delta)^q \le A^q - q \delta A^{q - 1}\] 

    Consequently, we have 
    \begin{align} \label{eq: bound for g - R_p f}
        \sum_{v \in V} \deg(v) (g(v) - R_p f(v))^q &= \lVert g - R_p f \rVert_q^q \nonumber \\ &\le (A - \delta)^q \le A^q - q\delta A^{q - 1} .
    \end{align} 
    
    By \eqref{eq: bound for g - R_p f}, the expectation of $\lVert f^\sharp - h \rVert_q$ can be estimated as follows:
    \begin{align} \label{ineq: expected norm of gsharp}
        &\E\left[ \lVert f^{\sharp} - h \rVert_q\right] \nonumber \\ = \ & \frac{1}{\lvert I \rvert}\sum_{v \in I}\Big(A^q - \deg(v)g(v)^q + \deg(v)(g(v) - R_p f(v))^q\Big)^\frac{1}{q} \nonumber \\ \le \ & \frac{1}{\lvert I \rvert} A^{1 - q}\sum_{v \in I}\left[ A^q - \deg(v) g(v)^q + \deg(v) \bigl(g(v) - R_p f(v)\bigr)^q\right] \nonumber \\ = \ & A - \frac{1}{\lvert I \rvert} A^{1 - q} \sum_{v \in V} \deg(v) \left( g(v)^q - (g(v) - R_p f(v))^q\right) \nonumber \\ \le \ & A - \frac{1}{n} A^{1 - q} \cdot q \delta A^{q - 1} \le A - qc_{2, p}^{1/q}n^{-2} \lVert g \rVert_\infty^{1/q} ,
    \end{align}

    where in the second step we use the elementary fact that
    $a^{1/q} \le a b^{1/q - 1}$ for $0 \le a \le b$ and $q \in (0, 1)$. This concludes Proposition~\ref{prop:main p<2}. 
\end{proof}

\subsection{\texorpdfstring{Estimate of relative energy decrease for $p \in [2, \infty)$}{Estimate of relative energy decrease for p>=2}} \label{sec: estimate of energy decrease}

With the help of Lemmas~\ref{Lemma: poincare inequality} and~\ref{lem: spectral gap}, we can prove Proposition~\ref{prop:main p>=2}. 

\begin{proof}[Proof of Proposition~\ref{prop:main p>=2}]
    We use the same notation as the proof of Proposition~\ref{prop:main p<2}. Recall the notation $g := f - h$ and  $U(x)=\sign(x)\lvert x \rvert^{p-1}$. By the definition of $p$-superharmonicity and Claim~\ref{claim: compare}, we also have $g \ge R_p f \ge 0$ as in case $p \in (1, 2)$. We also assume that $g \not \equiv 0$, since $g = 0$ case is trivial. 
    
    By \eqref{eq:pharmonic} applied to $p$-harmonic functions in $I$, for any $v \in I$, \eqref{eq: single vertex} also holds. In other words, for any $v \in I$, we have
    \begin{align*}
        &\sum_{w: w \sim v} \Bigl(U\big(f(v) - f(w)\big)- U(f(v) - R_p f(v) - f(w)) \Bigr)\nonumber \\ 
        = &\sum_{w: w \sim v} \Bigl(U\big(f(v) - f(w)\big) - U\big(h(v) - h(w)\big) \Bigr).
    \end{align*}
    
    Applying Claim~\ref{claim: elementary ineq 2} to the left-hand side \eqref{eq: single vertex}, we get that there exists a constant $C_{1, p} > 0$ such that for any $v \in I$ we have
    \begin{align} \label{ineq: p>=2 single vertex}
        &C_{1, p}\left(R_p f(v)\Bigl(\sum_{w: w \sim v} \lvert f(v) - f(w)\rvert^{p - 2}\Bigr) + \deg(v) (R_p f(v))^{p - 1}\right) \nonumber \\ \ge & \ \sum_{w: w \sim v} \Bigl(U \bigl(f(v) - f(w)\bigr) - U\bigl(h(v) - h(w)\bigr)\Bigr).
    \end{align}

    Similarly to the argument in the proof of Proposition~\ref{prop:main p<2}, by summing the inequality \eqref{ineq: p>=2 single vertex} over  all $v \in V$ with coefficient $g(v)$ and applying Claim~\ref{claim: elementary ineq 2}, we deduce that there exists a constant $c_{1, p} > 0$, such that
    \begin{align} \label{ineq: p>=2 global ineq}
        &C_{1, p} \sum_{v \in V} g(v)\Bigl(R_p f(v) \bigl(\sum_{w \sim v} \lvert f(v) - f(w) \rvert^{p - 2} \bigr) + \deg(v) (R_p f(v))^{p - 1}\Bigr) \nonumber \\ \ge \ &\sum_{v \in V} g(v) \sum_{w \sim v} \Bigl(U\big(f(v) - f(w)\big)  \nonumber - U\big(h(v) - h(w)\big)\Bigr)  \\ 
        = \ & \sum_{\{v, w \} \in E} \lvert (f(v) - h(v)) - (f(w) - h(w)) \rvert \cdot \nonumber  \left \lvert U\big(f(v) - f(w)\big) - U\big(h(v) - h(w)\big)\right \rvert \nonumber \\ \ge \ & c_{1, p} \sum_{\{ v, w \} \in E} \Bigl((g(v) - g(w))^2\lvert f(v) - f(w) \rvert^{p - 2} + \lvert g(v) - g(w) \rvert^p \Bigr), 
    \end{align}

    where we also use the monotonicity of function $U$ in the second-to-last step. 
    
    We define new variables $X_1, X_2, Y_1, Y_2$ as follows:
    \begin{align}
        X_1 &:= \sum_{v \in V} g(v) R_p f(v) \sum_{w \sim v} \lvert f(v) - f(w) \rvert^{p - 2} , \\ X_2 &:= \sum_{v \in V} \deg(v) g(v) (R_p f(v))^{p - 1}  , \\ Y_1 &:= \sum_{\{ v, w\} \in E} (g(v) - g(w))^2 \lvert f(v) - f(w) \rvert^{p - 2} , \\ Y_2 &:= \en_p(g) = \sum_{\{v, w \} \in E} \lvert g(v) - g(w) \rvert^p .
    \end{align}

    So the formula \eqref{ineq: p>=2 global ineq} can be simply written as 
    \begin{align} \label{ineq: abbrev p>=2 global ineq}
        X_1 + X_2 \ge \frac{c_{1, p}}{C_{1, p}}(Y_1 + Y_2) .
    \end{align}
    
    By the definition of random update, we have
    \begin{align} \label{eq: def of energy decrease}
        \E[\en_p(f)-\en_p(f^\sharp)] = \frac{1}{\lvert I \rvert}\sum_{v \in I}\left[ \sum_{w \sim v} \lvert f(v) - f(w) \rvert^p - \lvert f(v) - R_p f(v) - f(w) \rvert^p\right] .
    \end{align}

    Recall that by the definition of the local minimizer, for any $v \in I$, the following equation holds
    \begin{align} \label{eq: def of p Laplacian}
        \sum_{w \sim v} U(f(v) - R_p f(v) - f(w))  = 0 .
    \end{align}

    Therefore, by \eqref{eq:pharmonic}, \eqref{eq: def of p Laplacian} and Claim~\ref{claim: elementary ineq 3}, there exists a constant $c_{2,p} > 0$ such that for any fixed $v \in I$, we have
    \begin{align}\label{sec3.4:eqn1}
        & \sum_{w: w \sim v} \Big(\lvert f(v) - f(w) \rvert^p - \lvert f(v) - R_p f(v) - f(w) \rvert^p \Big) \nonumber \\
        = \ &\sum_{w: w \sim v} \Big(\lvert f(v) - f(w) \rvert^p - \lvert f(v) - R_pf(v) - f(w) \rvert^p - p \nonumber \\ & U(f(v) - R_p f(v) - f(w)) R_p f(v) \Big) \nonumber \\
        \ge \ &c_{2,p}\Big(R_pf(v)^2\sum_{w: w \sim v} \lvert f(v) - f(w) \rvert^{p - 2} + \deg(v) R_p f(v)^p\Big) .
    \end{align}
    
    Combining \eqref{eq: def of energy decrease} and \eqref{sec3.4:eqn1}, it follows that
    \begin{align} \label{ineq: energy decrease}
        \E[\en_p(f)-\en_p(f^\sharp)] &= \frac{1}{\lvert I \rvert}\sum_{v \in I} \sum_{w \sim v} \Bigl(\lvert f(v) - f(w) \rvert^p - \lvert f(v) - R_p f(v) - f(w) \rvert^p \Bigr)\nonumber \\ & \ge \frac{c_{2, p}}{n} \Big( \sum_{\{ v, w\} \in E} \lvert f(v) - f(w) \rvert^{p - 2} \left(R_p f(v)^2 + R_p f(w)^2 \right) \nonumber \\ &+ \sum_{v \in V} \deg(v) R_p f(v)^p\Big) .
    \end{align}

    We define two more variables $Z_1, Z_2$ as follows:
    \begin{gather}
        Z_1 := \sum_{\{ v, w \}\in E} \lvert f(v) - f(w) \rvert^{p - 2} \big(R_p f(v)^2 + R_p f(w)^2\big) , \\ Z_2 := \lVert R_p f \rVert_{p}^p = \sum_{v \in V} \deg(v) R_p f(v)^p .
    \end{gather}

    Then the formula \eqref{ineq: energy decrease} can be simply written as
    \begin{align} \label{ineq: abbrev energy decrease}
        \E[\en_p(f)-\en_p(f^\sharp)] \ge \frac{c_{2, p}}{n}(Z_1 + Z_2) .
    \end{align}
    On the other hand, by \eqref{eq:pharmonic},  
    \begin{align*}
        \en_p(f) - \en_p(h) &= \sum_{\{v, w \} \in E} \Big[\lvert f(v) - f(w) \rvert^p - \lvert h(v) - h(w) \rvert^p \\ &- p U\big(h(v) - h(w)\big)\Big((f(v) - f(w)) - (h(v) - h(w))\Big) \Big] .
    \end{align*}

    In conjunction with Claim~\ref{claim: elementary ineq 3}, we see that there exist two constants $c_{3, p}, C_{2, p} > 0$ such that
    \begin{align}
        &c_{3, p}(Y_1 + Y_2) \le \en_p(f) - \en_p(h) \label{ineq: lower bound for energy decrease} , \\ &\en_p(f) - \en_p(h) \le C_{2, p}(Y_1 + Y_2) \label{ineq: upper bound for energy decrease}.
    \end{align}

    By \eqref{ineq: abbrev p>=2 global ineq}, if we set $c_{4, p} := \frac{c_{1, p}}{2C_{1, p}}$, then we have: either  $X_1 \ge c_{4, p}(Y_1 + Y_2)$ or $X_2 \ge c_{4, p}(Y_1 + Y_2)$.
    We will deal with these two cases separately as follows: 

    {\bf Case 1: }
    
        \begin{align} \label{ineq: p>=2 case1}
            X_1 \ge c_{4, p}(Y_1 + Y_2)
            .
        \end{align}
    
        Without loss of generality, we can assume $V = \{ 1, 2, \cdots, n \}$ and $f: V \to \R$ is non-decreasing with respect to the index of vertices. Formally, we assume that $f(j) \le f(j+1)$ for any $1 \le j<n$. We define $x_i := g(i)$, $\alpha = \frac{p - 2}{p - 1}$ and \[a_{i, j} := \begin{cases}
            (f(j) - f(i))^{p - 1} & i < j  \land i \sim j , \\ 0 & \text{otherwise.} 
        \end{cases}
        \]

        By \eqref{eq:psuper}, for any $i \in I$, we must have \[ \sum_{j = 1}^{i - 1} a_{j, i} - \sum_{j = i + 1}^{n} a_{i, j} = \sum_{j \sim i} \sign(f(i) - f(j)) \lvert f(i) - f(j) \rvert^{p - 1} \ge 0.\] 

        Since $f$ is non-decreasing, we know that condition {\bf (d)} in Lemma~\ref{lem: spectral gap} holds for $\{a_{i, j}\}_{1 \le i < j \le n}$. Hence, by Lemma~\ref{lem: spectral gap}, we have
        \begin{align*} 
            Y_1 &= \sum_{\{ v, w\} \in E} (g(v) - g(w))^2 \lvert f(v) - f(w) \rvert^{p - 2} \\ &= \sum_{\substack{1 \leq i < j \leq n \\ i \sim j}} a_{i, j}^{\alpha} (x_j - x_i)^2 \\ &\ge \lambda_*(n, \alpha, D) \sum_{\substack{1 \leq i < j \leq n \\ i \sim j}} a_{i, j}^{\alpha}(x_i^2 + x_j^2) \\ &=  \lambda_*(n, \alpha, D) \sum_{v \in V} g(v)^2 \sum_{w \sim v} \lvert f(v) - f(w) \rvert^{p - 2} .
        \end{align*}

        Combining with the Cauchy-Schwarz inequality and \eqref{ineq: lower bound for energy decrease}, the variable $X_1$ can be bounded above as follows:
        \begin{align*}
            X_1 &\le Z_1^{1/2}\left( \sum_{v \in V} g(v)^2 \sum_{w \sim v} \lvert f(v) - f(w) \rvert^{p - 2}\right)^{1/2}  \nonumber \\ &\le \lambda_*(n, \alpha, D)^{-1/2} Y_1^{1/2}Z_1^{1/2} \\ &\le c_{3, p}^{-1/2} \lambda_*(n, \alpha, D)^{-1/2} Z_1^{1/2}\left(\en_p(f) - \en_p(h) \right)^{1/2} .
        \end{align*}

        Moreover, the inequality \eqref{ineq: upper bound for energy decrease} and \eqref{ineq: p>=2 case1} implies that \[ X_1 \ge \frac{c_{4, p}}{C_{2, p}}\left(\en_p(f) - \en_p(h) \right).\]

        Hence we get
        \begin{align}
            \frac{c_{4, p}}{C_{2, p}} \left( \en_p(f) - \en_p(h)\right) \le X_1 \leq c_{3, p}^{-1/2} \lambda_*(n, \alpha, D)^{-1/2}Z_1^{1/2} \left( \en_p(f) - \en_p(h)\right)^{1/2} .
        \end{align}

        It follows from the proof of Claim~\ref{energy minimizer} and Claim~\ref{claim: existence and uniqueness of harmonic extension} that $h$ is the unique minimizer of $\ell^p$ energy with prescribed boundary value $f \lvert_{B}$. Therefore $f \neq h$ implies that $\en_p(f) - \en_p(h) > 0$. 
        Hence, there exists a constant $c_{5, p} > 0$ such that 
        \begin{align*}
            Z_1 \ge c_{5, p} \lambda_*(n, \alpha, D) \left( \en_p(f) - \en_p(h) \right)
           .
        \end{align*}

        
        Combining with \eqref{ineq: abbrev energy decrease}, we get
        \begin{align}
            \E\left[\en_p(f) - \en_p(f^{\sharp})\right] &\ge \frac{c_{2, p}}{n} Z_1 \nonumber \\ &\ge \frac{c_{2, p}c_{5, p}}{n} \lambda_*(n, \alpha, D)\left(\en_p(f) - \en_p(h) \right) .
        \end{align}

        Proposition~\ref{prop:main p>=2} follows since for $\alpha = \frac{p - 2}{p - 1}$, we have $\frac{1}{n}\lambda_*(n, \alpha, D) \asymp_p F_{*}(n, p, D)$. 

        \medskip
        
    {\bf Case 2: } \begin{align} \label{ineq: p>=2 case2}
            X_2 \ge c_{4, p}(Y_1 + Y_2) .
        \end{align}

        By \hld's inequality, Lemma~\ref{Lemma: poincare inequality} and \eqref{ineq: lower bound for energy decrease}, the variable $X_2$ can be bounded above as follows:
        \begin{align}
            X_2 &\le \lVert g \rVert_p \lVert R_p f \rVert_p^{p - 1} \nonumber \\ &\le nD^{1/p} \en_{p}(g)^{1/p} \lVert R_p f \rVert_{p}^{p - 1} \nonumber \\ &= nD^{1/p} Y_2^{1/p}Z_2^{1-1/p} \nonumber \\ &\le c_{3, p}^{-1/p}nD^{1/p}Z_2^{1-1/p} \left( \en_p(f) - \en_p(h)\right)^{1/p}.
        \end{align}
    
        By \eqref{ineq: upper bound for energy decrease} and \eqref{ineq: p>=2 case2}, the variable $X_2$ can be bounded below as follows:
        \begin{align*}
            X_2 \ge \frac{c_{4, p}}{C_{2, p}} \left( \en_p(f) - \en_p(h)\right) .
        \end{align*}

        Hence, we have
        \begin{align}
            \frac{c_{4, p}}{C_{2, p}}\left(\en_p(f) - \en_p(h) \right) \le c_{3, p}^{-1/p}nD^{1/p} \left(\en_p(f) - \en_p(h) \right)^{1/p} Z_2^{1 - 1/p} .  
        \end{align}
        
        Similarly, we can assume that $\en_p(f) > 
        \en_p(h)$. Therefore, there exists a constant $c_{6, p} > 0$ such that
        \begin{align}
            Z_2 \ge c_{6, p} n^{-\frac{p}{p - 1}} D^{-\frac{1}{p - 1}}
            \left(\en_p(f) - \en_p(h) \right) .
        \end{align}

        Combining with  \eqref{ineq: abbrev energy decrease}, we have
        \begin{align}
            \E\left[\en_p(f) - \en_p(f^{\sharp})\right] &\ge \frac{c_{2, p}}{n} Z_2 \nonumber \\ &\ge \frac{c_{2, p}c_{6, p}}{n^{\frac{2p - 1}{p - 1}}D^{\frac{1}{p - 1}}} \left(\en_p(f) - \en_p(h) \right) .
        \end{align}

        It is easy to verify that $\frac{1}{n^{\frac{2p - 1}{p - 1}}D^{\frac{1}{p - 1}}} \ge F_{*}(n, p, D)$ when $D \in [1, n]$. Hence, Proposition~\ref{prop:main p>=2} also follows in this case. 
        
\end{proof}

\subsection{Proof of Theorem~\ref{th:main lpbound}: Upper bounds}
\label{proof of the main theorem}

By Remark~\ref{transfer to superharmonic}, we know that the theorem follows from the special case where $f_0$ is $p$-superharmonic in $I$. 

By Claim~\ref{claim: psuperharmonic keep} and Claim~\ref{claim: compare}, we know that $f_t$ is $p$-superharmonic in $I$ and $f_t \ge h$ for all $t \in \N$. Since $R_p f_{t - 1} \ge 0$, and the operation at time $t \in \Nstar$ lets $f_t(v_t) := f_{t - 1}(v_t) - R_pf_{t - 1}(v_t)$, we know that $\{f_t\}_{t \ge 0}$ is non-increasing. 

We define $g_t := f_t - h \ge 0$ in the following proof. 

\begin{proof}[Proof of Theorem~\ref{th:main lpbound} (a)]
    Let $\mathcal{F}_t$ be the $\sigma$-algebra generated by random variables $v_1, v_2, \cdots, v_t$. We define \[M_t := \lVert g_t \rVert_{p - 1} .\]
    
    For any $k \in \N$, we define $T_k := \tau_{p}(2^{-k})$ to be a sequence of stopping times, where $\tau_p(\eps)$ is the $\eps$-approximation time.

    For $k \ge 1$, we also define \[N_{t, k} := M_{t \land T_k}+ c_{1, p}n^{-2}(2^{-k})^{\frac{1}{p - 1}}\big((t \land T_k) - (t \land T_{k - 1})\big) ,\]
    
    where $c_{1, p}$ is the constant from Proposition~\ref{prop:main p<2}.

    Recall that by Proposition~\ref{prop:main p<2}, we have  \[ \E[ M_t - M_{t + 1} \mid \mathcal{F}_t] \ge c_{1, p}n^{-2} \lVert g_t \rVert_\infty^{\frac{1}{p-1}}.\] 

    Combining this with the non-increasing property of $M_t$, we infer that $N_{t, k}$ is a nonnegative supermartingale with respect to $\mathcal{F}_t$.  For every integer $L\ge1$, optional sampling at the bounded stopping times $T_{k-1}\wedge L\le T_k\wedge L$ gives
    \[
        \E[N_{T_k\wedge L,k}]\le \E[N_{T_{k-1}\wedge L,k}]
        =\E[M_{T_{k-1}\wedge L}].
    \]
    Since $0\le M_t\le(nD)^{1/(p-1)}$, dominated convergence applies to the right-hand side, while Fatou's lemma applies to the nonnegative left-hand side.  Letting $L\to\infty$ therefore yields
    \[
        \E[N_{T_k,k}]\le \E[M_{T_{k-1}}].
    \]
    This implies that
    \begin{align}
        \E[M_{T_k}] + c_{1, p} n^{-2} (2^{-k})^{\frac{1}{p - 1}} \E[T_k - T_{k - 1}] \le \E[M_{T_{k - 1}}] .
    \end{align}

    Hence $\E[T_k - T_{k - 1}]$ can be bounded as follows:
    \begin{align}
        \E[T_{k} - T_{k - 1}] &\le \frac{n^2}{c_{1, p}}(2^k)^{\frac{1}{p - 1}} \E[M_{T_{k - 1}}] \nonumber \\ &= \frac{n^2}{c_{1, p}}(2^k)^{\frac{1}{p - 1}} \E\Big[\Big(\sum_{v \in V} \deg(v) g_{T_{k-1}}(v)^{p - 1}\Big)^{\frac{1}{p - 1}}\Big] \nonumber \\ &\le \frac{n^2}{c_{1, p}}(2^k)^{\frac{1}{p - 1}}(nD)^{\frac{1}{p - 1}}2^{-k + 1} = \frac{2}{c_{1, p}}n^{\frac{2p - 1}{p - 1}}D^{\frac{1}{p - 1}}(2^k)^{\frac{2 - p}{p - 1}}  .
    \end{align}

    For a given $\epsilon \in (0, 1/2]$, let $K = \lfloor \log_2(\frac{1}{\epsilon}) \rfloor + 1$. Then there exists a constant $C_{1, p} > 0$ such that
    \begin{align}
        \E[\tau_{p}(\epsilon)] \le \E[T_{K}] &= \sum_{k = 1}^{K} \E[T_k - T_{k - 1}] \nonumber \\ &\le \sum_{k = 1}^{K} \frac{2}{c_{1, p}}n^{\frac{2p - 1}{p - 1}}D^{\frac{1}{p - 1}}(2^k)^{\frac{2 - p}{p - 1}} \nonumber \\ &\le C_{1, p}n^{\frac{2p - 1}{p - 1}}D^{\frac{1}{p - 1}} \epsilon^{-\frac{2-p}{p-1}} , 
    \end{align}

    which is exactly what we wanted (Note that we have $T_0 = 0$ since $\lVert f_0 - h \rVert_\infty \le 1$).  
\end{proof}

\begin{proof}[Proof of Theorem~\ref{th:main lpbound} (b)]
    We define \[ M_t := \en_p(f_t) - \en_p(h) .\]

    By Proposition~\ref{prop:main p>=2}, for any $t \in \N$, we have \[ \E[M_{t + 1} \mid \mathcal{F}_t] \le (1 - c_{2, p} F_{*}(n, p, D)) M_t ,\] 

    where $c_{2, p} > 0$ is the constant in Proposition~\ref{prop:main p>=2}. A simple induction implies that for every $t \in \N$, we have
    \begin{align}
        \E[M_t] &\le \big(1 - c_{2, p} F_{*}(n, p, D) \big)^t \E[M_0] \nonumber \\ &\le n^2\big(1 - c_{2, p} F_{*}(n, p, D)\big)^t .
    \end{align}

    By \eqref{ineq: lower bound for energy decrease} and Lemma~\ref{Lemma: poincare inequality}, there exists another constant $c_{3, p} > 0$ such that on the event that $\{\tau_{p}(\epsilon) > t \}$, we have
    \begin{align}
        M_t \ge c_{3, p} \sum_{\{v, w\} \in E} \lvert g_t(v) - g_t(w) \rvert^p \ge \frac{c_{3, p}}{n^pD} \sum_{v \in V} g_t(v)^p \ge \frac{c_{3,p}\epsilon^p}{n^pD} .
    \end{align}

    For any $t \in \N$, by Markov's inequality, we always have
    \begin{align}
        \Pr[\tau_{p}(\epsilon) > t] \le \frac{n^{p + 2}D\big(1 - c_{2, p} F_{*}(n, p, D)\big)^t}{c_{3, p}\epsilon^p} .
    \end{align}
    
    Decreasing $c_{2,p}$ if necessary, we may assume $0<c_{2,p}F_*(n,p,D)\le1/2$.  Define
    \[
        T_0 := \Big\lceil \frac{\log(n^{p + 2}D/(c_{3, p}\epsilon^p))}{c_{2, p} F_{*}(n, p, D)}\Big\rceil .
    \]
    Since $(1-x)^t\le e^{-xt}$ for $x\in(0,1)$, the preceding probability bound is at most $1$ at $t=T_0$ and thereafter. There exists a constant $C_{2, p} > 0$ such that

    \begin{align}
        \E[\tau_{p}(\epsilon)] &\le T_0 + \sum_{t \ge T_0} \Pr[\tau_p(\epsilon) > t] \nonumber \\ &\le T_0 + \sum_{t \ge T_0} \frac{n^{p + 2}D\big(1 - c_{2, p} F_{*}(n, p, D)\big)^t}{c_{3, p}\epsilon^p} \nonumber \\&\le T_0 + 2\big(c_{2, p} F_{*}(n, p, D)\big)^{-1} \le C_{2, p} F_{*}(n, p, D)^{-1} \log(n/\epsilon) .
    \end{align}
\end{proof}

\subsection{Proof of Corollary~\ref{cor: main cor}: upper bound}
\label{sec:synchronous-upper}

Recall that in the synchronous $\ell^p$-relaxation dynamics, the opinions at interior vertices are updated simultaneously,  see \eqref{eq:synch}.

\paragraph{\bf{Case 1}: $p\in(1,2)$}

By Proposition~\ref{prop: synchronous superharmonic reduction}, it suffices to prove the upper bound for $p$-superharmonic initial profiles.  Proposition~\ref{prop: synchronous superharmonic reduction} also shows that this property is preserved under every synchronous step, so the following one-step estimate may be iterated:

\begin{claim}
    Fix $p \in (1, 2)$. Suppose $f_0^*:V \to [0,1]$ is $p$-superharmonic in $I$, and $h$ is the $p$-harmonic extension with respect to $f_0^*\vert_B$. Then there exists a constant $c_p>0$ such that
    \begin{align}
        &\Big(\sum_{v\in V}\deg(v)(f^*_1(v)-h(v))^{p-1}\Big)^{\frac{1}{p-1}} \nonumber \\
        &\le \Big(\sum_{v\in V}\deg(v)(f^*_0(v)-h(v))^{p-1}\Big)^{\frac{1}{p-1}}-c_pn^{-1}\lVert f^*_0-h\rVert_\infty^{\frac{1}{p-1}}.
    \end{align}
\end{claim}

The idea of proving this claim is almost the same as for Proposition~\ref{prop:main p<2}.

Let $g := f_0^* - h$. We can assume $\lVert g \rVert_{p - 1} \neq 0$, since $g = 0$ case is trivial. 

Recall that in the proof of Proposition~\ref{prop:main p<2}, we have shown that (see equation~\eqref{eq: bound for g - R_p f}),
\begin{align*}
    \sum_{v \in V} \deg(v) (g(v) - R_p f_0^*(v))^{p - 1} \le \lVert g \rVert_{p - 1}^{p - 1} - c_{3, p} n^{-1} \lVert g \rVert_\infty^{\frac{1}{p - 1}} \lVert g \rVert_{p - 1}^{p - 2} ,
\end{align*}

for some constant $c_{3,p}>0$.  This is a deterministic inequality for the vector $g-R_pf_0^*$, before the random choice of an update vertex is made.  Under a synchronous step we have, for every $v\in I$,
\[
    f_1^*(v)-h(v)=g(v)-R_pf_0^*(v),
\]
and $R_pf_0^*(v)=0$ on $B$ by convention.  Hence exactly the same inequality applies to the simultaneous update.

Therefore, we have
\begin{align*}
    \Big(\sum_{v\in V}\deg(v)(f_1^*(v)-h(v))^{p-1}\Big)^{\frac{1}{p-1}} &= \Big(\sum_{v\in V}\deg(v)(g(v)-R_pf_0^*(v))^{p-1}\Big)^{\frac{1}{p-1}} \\
    &\le \lVert g \rVert_{p - 1} - c_{3, p} n^{-1} \lVert g \rVert_\infty^{\frac{1}{p - 1}} \,.
\end{align*}
This completes the proof.

\paragraph{\bf{Case 2:} $p\in[2,\infty)$}

Let $G=(V = I \bigsqcup B, E)$ be the original graph. Denote $\lvert V \rvert$ and $\lvert I \rvert$ by $n$ and $m$, respectively.  Let $V_+, V_-$ be two copies of $V$. Let $\phi_+: V \to V_+$ and $\phi_-: V \to V_-$ be the natural bijection between $V\to V_+$ and $V \to V_-$.  We construct a bipartite graph $H$, with vertex set $V(H)=V_+\bigsqcup V_-$ and edge set 
\begin{align}
    E(H)=\{\{\phi_+(v), \phi_-(w)\}:\{v, w\}\in E\} \,.
\end{align}
\begin{figure}
    \centering
    \includegraphics[width=0.5\linewidth]{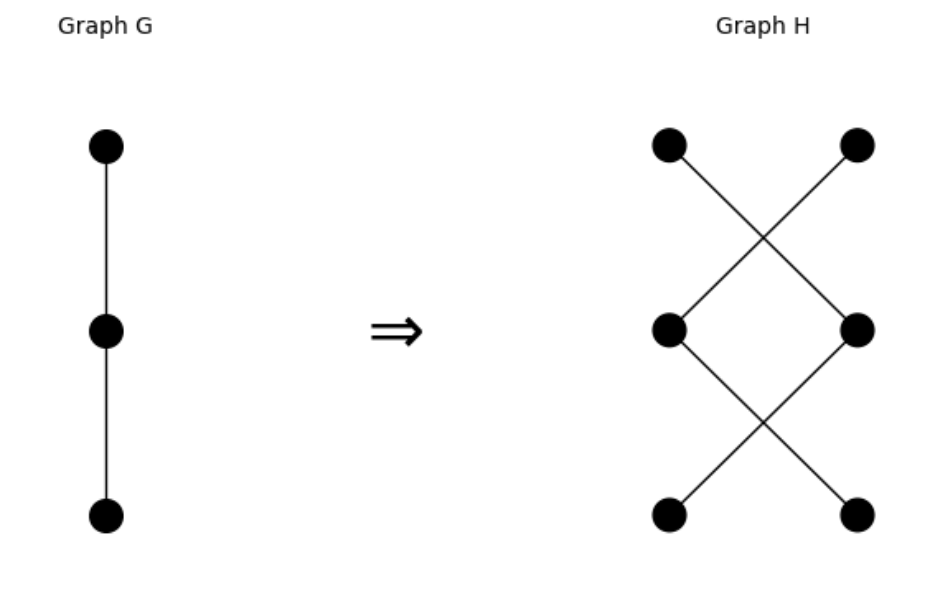}
    \caption{Bipartite double cover of $G$}
\end{figure}

The interior vertices of $H$ are $I(H)=\phi_+(I) \bigsqcup \phi_-(I)$.

Given the initial profile $f_0^*$ on $G$, consider the following initial profile $g_0$ on $H$:
\begin{align}
    g_0(\phi_+(v))=g_0(\phi_-(v))=f_0^*(v)\quad\forall v \in V .
\end{align}

Let $f_t^*$ be the opinion profile on $G$ at time $t$ under the synchronous dynamics. Let $\tau_G^{sync}$ be the $\eps$-approximation time of graph $G$ under synchronous update.

By Proposition~\ref{prop: synchronous superharmonic reduction}, we may assume that $f_0^*$ is $p$-superharmonic in $I$; then $g_0$ is $p$-superharmonic in $I(H)$.  Let $h$ be the $p$-harmonic extension on $G$ and define its lift $\widetilde h$ to $H$ by
\[
    \widetilde h(\phi_+(v))=\widetilde h(\phi_-(v))=h(v)\qquad(v\in V).
\]
Then $\widetilde h$ is the $p$-harmonic extension of the boundary values of $g_0$ on $H$.

Consider the cyclic update on graph $H$, the update sequence of the vertices is defined as follows: For every $k\in \Nstar$,
\begin{itemize}
    \item at stage $2k-1$, update the vertices in $\phi_+(I)$ sequentially,
    \item at stage $2k$, update the vertices in $\phi_-(I)$ sequentially.
\end{itemize}
By Claim~\ref{claim: psuperharmonic keep} and an induction argument, we obtain that the opinion profile on $H$ is always $p$-superharmonic in $I(H)$.

Note that $H$ is a bipartite graph. It follows by induction that after stage $\ell$, the opinion profile on $H$ is 
\begin{align}
    g_\ell(\phi_+(v))=f_\ell^*(v), \quad g_\ell(\phi_-(v))=f_{\ell-1}^*(v),
\end{align}
for all $v\in V$, if $\ell$ is odd. The opinion profile on $H$ is
\begin{align}
    g_\ell(\phi_+(v))=f_{\ell-1}^*(v), \quad g_\ell(\phi_-(v))=f_{\ell}^*(v),
\end{align}
for all $v\in V$, if $\ell$ is even.

Let $\tau^{cyc}_H$ be the $\eps$-approximation time in the graph $H$ under this cyclic update, measured in individual vertex updates (so each stage consists of $m$ updates). During a stage, the copy not being updated is frozen. Thus, if $\tau_G^{sync}=t$, the cyclic process cannot reach $\epsilon$-accuracy before the end of stage $t$, and it is certainly $\epsilon$-accurate by the end of stage $t+1$. Therefore
\begin{align}\label{sync2cyc}
    m\tau_G^{sync}\le\tau^{cyc}_H\le m(\tau_G^{sync}+1).
\end{align}

Next, we consider the energy decrease of the cyclic update at stage $\ell+1$. Without loss of generality, assume $\ell$ is odd. Therefore, $g_\ell(\phi_+(v))=f_\ell^*(v)$ and $g_\ell(\phi_-(v))=f_{\ell-1}^*(v)$. The vertices in $\phi_+(I)$ were all updated in the preceding stage and are local minimizers against the current $\phi_-$-values, so a hypothetical update at any of them produces zero energy decrease. Since there are no edges within $\phi_-(I)$, the updates in the next stage commute and their energy decreases add.

Thus we have
\begin{align}
    \sum_{u\in I(H)} \bigl(\en_p(g_{\ell})-\en_p(g^{\#u}_\ell) \bigr)=\sum_{v \in I} \bigl(\en_p(g_\ell)-\en_p(g_\ell^{\#{\phi_-(v)}}) \bigr),
\end{align}
where $g^{\#u}_\ell$ represents the function after updating the value at vertex $u$.
Notice that the right-hand side is exactly the total energy decrease at stage $\ell+1$ under the cyclic update, since $\en_p(g_\ell)-\en_p(g_\ell^{\#{\phi_-(v)}})$ only depends on the profile on $\phi_-(v)$ and its neighbors.

If $H$ is connected, then we can now apply Proposition~\ref{prop:main p>=2} on $g_\ell$. It follows that there exists a constant $c_p > 0$ such that
\begin{align}
    \frac{1}{2m}\sum_{v\in I} \bigl(\en_p(g_\ell)-\en_p(g_{\ell}^{\#{\phi_-(v)}}) \bigr) = &\ \frac{1}{2m}\sum_{v\in I(H)} \bigl(\en_p(g_\ell) -\en_p(g^{\#v}_\ell) \bigr)\nonumber \\
    \ge& \ c_pF_*(2n,p,D_{H})(\en_p(g_\ell)-\en_p(\widetilde h)) \nonumber \\
    =& \ c_pF_*(2n,p,D_G)(\en_p(g_\ell)-\en_p(\widetilde h)).
\end{align}

If $H$ is not connected, then $H$ contains exactly two connected components, each connected component is of size $n$ and average degree $D_G$. Therefore, by Proposition~\ref{prop:main p>=2}, we have
\begin{align}
    \frac{1}{2m} \sum_{v\in I} \bigl(\en_p(g_\ell)-\en_p(g_{\ell}^{\#{\phi_-(v)}} ) \bigr)= &\ \frac{1}{2m}\sum_{v\in I(H)} \bigl(\en_p(g_\ell)-\en_p(g^{\#v}_\ell) \bigr)\\
    \ge& \ \frac{c_p}{2}F_*(n,p,D_G)(\en_p(g_\ell)-\en_p(\widetilde h)).
\end{align}

Finally, using the same argument as in Section~\ref{proof of the main theorem}, we can conclude that
\begin{align}
    \tau_H^{cyc}\le \frac{C_p\log(n/\eps)}{F_*(n,p,D_G)},
\end{align}
for some $C_p>0$. Combining with \eqref{sync2cyc}, we complete the proof of this corollary.

\section{Proof of Theorem~\ref{th:main lpbound} and Corollary~\ref{cor: main cor}: lower bounds} \label{sec:lowerbounds}

\subsection{Lower bounds using the number of vertices and average degree}
Amir, Nazarov and Peres proved lower bounds for the consensus time without boundary in~\cite{noboundarycase}.  We first transfer their fixed-threshold lower bound to the setting $B\ne\varnothing$.  We then prove the new joint lower bound in $N$, $D$, and $\epsilon$ for $1<p<2$.

\begin{thm}[Amir--Nazarov--Peres] \label{th: no boundary lower bound}    
    There exist constants $c_p, C_p > 0$ with the following property: for every $N \ge C_p$ and every $2 \le D \le N$, there is a connected graph $G = (V, E)$ with $\lvert V \rvert \le N$, $D_G \le D$ and initial profile $f_0: V \to [0, 1]$, such that for any update sequence $\{ v_t\}_{t \in \Nstar}$ and $t \le c_pN^{\beta_p}(\frac{D}{N})^{\theta_p}$, we have \[ \max_{v \in V} f_t(v) - \min_{v \in V} f_t(v) > \frac{1}{2} .\]
\end{thm}

Recall that $\beta_p = \max \bigl\{\frac{2p}{p - 1}, 3\bigr\}$ and $\theta_p$ is defined in \eqref{eq:thetadef}.  This is the converse part of~\cite[Theorem~1.9]{noboundarycase}; the deterministic form for arbitrary update sequences is the form established by the lower-bound constructions in Section~5 of that paper.  We quote it here in exactly the form needed below.

Theorem~\ref{th: no boundary lower bound} implies the following fixed-threshold result when $B\ne\varnothing$.
\begin{thm} \label{th: with boundary lower bound}
    There exist constants $c_p, C_p > 0$ with the following property: for every $N \ge C_p$ and every $2 \le D \le N$, there is a connected graph $G = (V = I \bigsqcup B, E)$ with $\lvert V \rvert \le N$, $D_G \le D$ and initial profile $f_0: V \to [0, 1]$, such that for any update sequence $\{ v_t\}_{t \in \Nstar}$ with $v_t \in I$ and $t \le c_pN^{\beta_p}(\frac{D}{N})^{\theta_p}$, we have \[ \lVert f_t - h \rVert_\infty> \frac{1}{2} ,\]
    where $h$ is the $p$-harmonic extension of $f_0\vert_{B}$. Therefore, for any update sequence, we have $\tau_p(1/2)\ge c_pN^{\beta_p}(D/N)^{\theta_p}$.
\end{thm}

\begin{proof}
    Let $G = (V, E)$ and initial profile $f_0: V \to [0, 1]$ be the desired instance given in Theorem~\ref{th: no boundary lower bound}. Let $f_0(v_0)$ be the minimum value of $f_0(v)$ for all $v \in V$. We set $B = \{v_0 \}$ and $I = V \setminus B$. The condition in Theorem~\ref{th: no boundary lower bound} implies that for any update sequence $\{v_t \}_{t \in \N^{*}}$ with $v_t \in I$ and $t \le c_pN^{\beta_p}(\frac{D}{N})^{\theta_p}$, we have \[ \max_{v \in V} f_t(v) - \min_{v \in V} f_t(v) > \frac{1}{2} .\]

    Compare the dynamics with the constant profile $\underline f\equiv f_0(v_0)$.  Since $f_0\ge\underline f$ and an update leaves $\underline f$ unchanged, Claim~\ref{claim-mono} and induction on $t$ give $f_t\ge\underline f$ for all $t$. Moreover, $f_t(v_0) = f_0(v_0)$ for any $t \in \N$, since we never update the boundary vertex $v_0$. Note that the $p$-harmonic extension of $f_0 \vert_B$ is exactly the constant function with value $f_0(v_0)$. 

    Hence we have \[ \lVert f_t - h \rVert_\infty = \max_{v \in V} f_t(v) - \min_{v \in V} f_t(v) > \frac{1}{2}.\]
\end{proof}

\begin{rem}
    We record the transfer from asynchronous to synchronous lower bounds.  Suppose that a graph and a profile $f_0$ satisfy $\tau_p(\delta)\ge T$ for every asynchronous update sequence in $I$.  Fix a round-robin sequence that visits every vertex of $I$ once per round.
    
    We define $f_0^+, f_0^-: V \to [0, 1]$ as follows:
    \begin{align}
        f_0^+(v) := \begin{cases}
            1 & v \in I , \\ f_0(v) & v \in B .
        \end{cases} \\ f_0^-(v) := \begin{cases}
            0 & v \in I , \\ f_0(v) & v \in B . 
        \end{cases} 
    \end{align}
    
    Let $\tau_p^+(\delta)$ and $\tau_p^-(\delta)$ be the approximation times for the profiles $f_0^+$ and $f_0^-$ along this update sequence.  Both profiles have the same $p$-harmonic extension $h$ as $f_0$.
    
    Monotonicity gives $f_t^-\le f_t\le f_t^+$, and therefore
    \[ \tau_p(\delta) \le \max\{\tau_p^+(\delta), \tau_p^-(\delta)\}. \]
    
    Hence either $f_0^+$, or $1-f_0^-$ with harmonic extension $1-h$, is $p$-superharmonic in $I$ and has asynchronous approximation time at least $T$ along the round-robin sequence.  Denote the resulting superharmonic profile again by $f_0$ and its harmonic extension by $h$.

    Let $\{f_t\}_{t \in \N}$ be the round-robin asynchronous process and $\{f_t^*\}_{t \in \N}$ the synchronous process from this profile.  Monotonicity and induction give
    $f_t^*\ge f_{\lvert I\rvert t}\ge h$ for all $t\in\N$.
    
    Thus $\lVert f_t^*-h\rVert_\infty>\delta$ whenever $\lvert I\rvert t<T$, and
    \[
        \tau_p^{sync}(\delta)\ge \frac{T}{\lvert I\rvert}.
    \]
    Applying this observation to Theorem~\ref{th: joint lower p<2} for $p\in(1,2)$ and to Theorem~\ref{th: with boundary lower bound} for $p\ge2$ proves the lower bounds in Corollary~\ref{cor: main cor}, after adjusting the constants to account for integer parts.
    
\end{rem}

\subsection{\texorpdfstring{Lower bounds with respect to $\epsilon$}{Lower bounds with respect to epsilon}} \label{subsec:epsilon}

We first isolate the slow small-error mechanism on a fixed graph.  The subsequent subsection combines it with the dependence on graph size and average degree.

Take $G =K_4$, and name the vertices of $G$ by $V := \{a, b, w_1, w_2 \}$. Let $B = \{a, b \}$ and $I = \{w_1, w_2 \}$. Consider the initial profile  
\[ f_0(v) = {\bf 1}_{v \neq a} .\]

Let $h$ denote  the $p$-harmonic extension of $f_0 \vert_{B}$. It is easy to check that  
\[ h(v) = \begin{cases} 0 & v = a ,\\ \frac{1}{2} & v \in \{w_1, w_2 \} ,\\ 1 & v = b .
    \end{cases} \]

The following four-vertex example was observed in \cite[Question 6.2]{noboundarycase}; we give a detailed proof because its mechanism enters the joint construction below.

\begin{prop} \label{powerlower}
    For all $p \in (1, 2)$, there exists $c_p > 0$ with the following property: for the graph $G=K_4$ with initial profile $f_0$ defined above, any update sequence $\{v_t \}_{t \in \Nstar}$ in $I$ and any $\epsilon \in (0, 1/4)$, we have
    \begin{align}
        \tau_p(\epsilon) \ge c_p \epsilon^{-\frac{2-p}{p-1}} .
    \end{align}
\end{prop}
\begin{proof}

    Let $\{ v_t \}_{t \in \N^{*}}$ be any update sequence in $I$.  Delete consecutive repetitions, which do not change the profile after the first update in each run.  Up to the first time at which the error falls below $\epsilon$, the resulting effective updates alternate between the two interior vertices; if the effective updates stop earlier, the profile remains fixed and the claimed lower bound is immediate.  Relabeling $w_1,w_2$ if necessary, we can therefore write the effective sequence as $v_{2k-1}=w_1$ and $v_{2k}=w_2$.  A lower bound for compressed time is also a lower bound for the original time.

    Let $a_0 = 1$ and $a_t$ be the value obtained by updating $v_t$ from $f_{t - 1}$. Define $b_t := a_t - \frac{1}{2}$. Then we must have $\lVert f_t - h \rVert_\infty \ge b_t$. It is also easy to see that $a_t \in (\frac{1}{2}, a_{t - 1})$ and 
    \begin{align} \label{eq: recursive formula of a_t}
        (a_{t - 1} - a_t)^{p - 1} = a_t^{p - 1} - (1 - a_t)^{p - 1} .
    \end{align}

    Since $x \mapsto\frac{x^{p - 1} - (1 - x)^{p - 1}}{x - \frac{1}{2}}$ extends continuously to $[\frac{1}{2}, 1]$, there exists $C_{1, p} > 0$ such that
    \begin{align} \label{transfer to inequality}
        (a_{t - 1} - a_t)^{p - 1} \le C_{1, p}(a_t - \frac{1}{2}) .
    \end{align}
    
    The inequality \eqref{transfer to inequality} is equivalent to the following inequality
    \[ b_{t - 1} \le b_t\left(1 + C_{1, p}^{\frac{1}{p - 1}} b_t^{\frac{2 - p}{p - 1}}\right) . \]

    Put $\rho:=(2-p)/(p-1)>0$.  Choose a constant $K_p\ge 2^\rho$ large enough that
    \[
        (1+C_{1,p}^{1/(p-1)}x)^{-\rho}\ge1-K_px
        \qquad(0\le x\le1).
    \]
    Since the exponent $-\rho$ is negative, the preceding inequality for $b_{t-1}$ gives
    \begin{align}
        b_{t-1}^{-\rho}\ge b_t^{-\rho}-K_p.
    \end{align}
    Iterating and using $b_0=1/2$ therefore yields
    \[
        b_t^{-\rho}\le b_0^{-\rho}+K_pt=2^\rho+K_pt\le K_p(t+1).
    \]
    Equivalently, $b_t\ge (K_p(t+1))^{-1/\rho}$. Since $\lVert f_t-h\rVert_\infty\ge b_t$, for some constant $c_p>0$ and every $\epsilon\in(0,1/4)$ we obtain
    \[ \tau_{p}(\epsilon) \ge \min \{t \ge 0:  b_t \le \epsilon\} \ge c_p\epsilon^{-\frac{2-p}{p-1}} .\]
    
\end{proof}

\subsection{\texorpdfstring{A joint lower bound in $n$, $D$, and $\epsilon$}{A joint lower bound in n, D, and epsilon}}

The preceding four-vertex example isolates the slow small-error mechanism, but does not combine it with the graph-size obstruction. We now combine that mechanism with the clique reservoir used in the dense lower-bound construction of~\cite{noboundarycase}. The two long arms below make the restoring force at the central vertex of order $L^{1-p}$, and the cliques delay the response of its neighbors. These two effects give exactly the product appearing in the upper bound.

Fix $p\in(1,2)$ and write
\begin{align*}
    q:=p-1\in(0,1),
    \qquad
    \rho:=\frac{1-q}{q}=\frac{2-p}{p-1}.
\end{align*}
For integers $L,m,d\ge1$, let $G_{L,m,d}$ be the following graph (see Figure~\ref{fig: joint lower graph}). Start with the path
\begin{align*}
    x_{-L}\sim x_{-L+1}\sim\cdots\sim x_0\sim\cdots\sim x_{L-1}\sim x_L.
\end{align*}
For each $1\le i\le m$, add a clique
$W_i=\{w_{i,1},\ldots,w_{i,d}\}$ and join every vertex of $W_i$ to $x_0$. Take
\begin{align*}
    B:=\{x_{-L},x_L\},
    \qquad f_0(x_{-L})=0,
    \qquad f_0(x_L)=1.
\end{align*}
The $p$-harmonic extension of these boundary data is
\begin{align}\label{eq: joint lower harmonic profile}
    h(x_j)=\frac{j+L}{2L},
    \qquad
    h(w)=\frac12\quad(w\in W_1\cup\cdots\cup W_m).
\end{align}
Put $M:=md$ and define the tent function $\psi$ by
\begin{align*}
    \psi(x_j):=1-\frac{|j|}{L},
    \qquad
    \psi(w):=1\quad(w\in W_1\cup\cdots\cup W_m).
\end{align*}

\begin{figure}[t]
\centering
\begin{tikzpicture}[
    x=1.15cm, y=1cm,
    int/.style={circle, fill=black, inner sep=0pt, minimum size=4pt},
    bdry/.style={rectangle, fill=black, inner sep=0pt, minimum size=5pt},
    clique/.style={circle, draw=black, fill=black!20, inner sep=0pt, minimum size=8mm},
    every node/.style={font=\small},
    every label/.style={font=\small}
  ]
  \node at (-0.2,2.55) {(a)};
  \draw (0,0) -- (1,0) -- (1.45,0);
  \node at (2,0) {$\cdots$};
  \draw (2.55,0) -- (3,0) -- (7,0) -- (7.45,0);
  \node at (8,0) {$\cdots$};
  \draw (8.55,0) -- (9,0) -- (10,0);
  \foreach \cx in {3, 4.35, 7} {
    \draw[thin] (5,0) -- (\cx-0.28,1.9);
    \draw[thin] (5,0) -- (\cx,1.9);
    \draw[thin] (5,0) -- (\cx+0.28,1.9);
  }
  \node[clique, label=above:{$W_1$}] at (3,1.9) {$K_d$};
  \node[clique, label=above:{$W_2$}] at (4.35,1.9) {$K_d$};
  \node at (5.65,1.9) {$\cdots$};
  \node[clique, label=above:{$W_m$}] at (7,1.9) {$K_d$};
  \node[bdry, label=below:{$x_{-L}$}] at (0,0) {};
  \node[int, label=below:{$x_{-L+1}$}] at (1,0) {};
  \node[int, label=below:{$x_{-2}$}] at (3,0) {};
  \node[int, label=below:{$x_{-1}$}] at (4,0) {};
  \node[int, label=below:{$x_{0}$}] at (5,0) {};
  \node[int, label=below:{$x_{1}$}] at (6,0) {};
  \node[int, label=below:{$x_{2}$}] at (7,0) {};
  \node[int, label=below:{$x_{L-1}$}] at (9,0) {};
  \node[bdry, label=below:{$x_{L}$}] at (10,0) {};
  \node[above=3pt] at (0,0) {$f_0=0$};
  \node[above=3pt] at (10,0) {$f_0=1$};
  \node[right=2pt, align=left] at (7.55,1.9) {$f_0\equiv\tfrac12+2\epsilon$\\ on every $W_i$};

  \begin{scope}[yshift=-3.1cm]   
    \node at (-0.2,1.85) {(b)};
    \draw[gray!60, thin] (0,0) -- (10,0);
    \draw[gray!60, thin] (0,1.5) -- (10,1.5);
    \foreach \x in {0,1,3,4,5,6,7,9,10} { \draw[gray!60, thin] (\x,-0.06) -- (\x,0.06); }
    \node[left=2pt] at (0,0) {$0$};
    \node[left=2pt] at (0,1.5) {$1$};
    \draw[dashed] (0,0) -- (10,1.5) node[pos=0.78, below, sloped] {$h$};
    \draw[thick] (0,0) -- (5,0.93) node[pos=0.5, above, sloped] {$f_0=h+2\epsilon\psi$} -- (10,1.5);
    \draw[dotted] (5,0.75) -- (5,2.6);
    \draw[fill=white] (5,0.75) circle (1.6pt);
    \fill (5,0.93) circle (1.6pt);
    \node[above left=-1pt] at (5,0.93) {$\tfrac12+2\epsilon$};
    \node[below right=-1pt] at (5,0.75) {$\tfrac12$};
  \end{scope}
\end{tikzpicture}
\caption{The graph $G_{L,m,d}$, which provides the lower bound in Theorem~\ref{th: joint lower p<2}. (a) A path of length $2L$ whose endpoints $x_{\pm L}$ (squares) form the boundary, with boundary values $0$ and $1$. The shaded discs are the cliques $W_1,\ldots,W_m$ of size $d$; every vertex of every clique is joined to the central vertex $x_0$. (b) The $p$-harmonic extension $h$ (dashed) and the initial profile $f_0=h+2\epsilon\psi$ of Lemma~\ref{lem: joint pinned reservoir} (solid) along the path. The profile $f_0$ equals $\frac12+2\epsilon$ at $x_0$ and on every clique; the height $2\epsilon$ of the tent is exaggerated.}
\label{fig: joint lower graph}
\end{figure}

\begin{lem} \label{lem: joint pinned reservoir}
    There exist constants $C_p, c_p > 0$ such that the following holds: If $d \ge 1$, $L, M \ge C_p$ and $\epsilon \in (0,1/16]$, then the initial profile 
    \begin{align} \label{eq: joint lower initial profile}   
        f_0(v) := h(v) + 2\epsilon\psi(v) \quad (v \in I)
    \end{align}
    
    takes values in $[0,1]$ and, for every sequence of update vertices in $I$,
    \begin{align}\label{eq: parameterized joint lower}
        \tau_p(\epsilon)
        \ge c_p L M^{1+1/q}d^{1/q}\epsilon^{-\rho}.
    \end{align}
\end{lem}

\begin{proof}
    We first check the range of the initial profile.  On the left half of the path, $\psi=2h$, so $f_0=(1+4\epsilon)h\in[0,1]$.  On the right half, $\psi=2(1-h)$, so $f_0=h+4\epsilon(1-h)\in[0,1]$.  At every clique vertex, $f_0=1/2+2\epsilon\in[0,1]$.  Here we used $\epsilon\le1/16$; the boundary values remain fixed.  We prove the time bound by a barrier argument.

    The update rule commutes with reflection about $1/2$, so $z_t:=1-f_t$ follows the same $\ell^p$-relaxation dynamics, with boundary values $1$ and $0$.

    For $r, a, \eta > 0$, we define an upper envelope $H_{r, a, \eta}$ by 
    \begin{align} \label{eq: joint lower envelope}
        H_{r, a, \eta}(x_j) &:= 1 - h(x_j) - (r-a)\psi(x_j) , \nonumber \\
        H_{r, a, \eta}(w_{i,j}) &:= \frac{1}{2} - r + (j - 1) \eta .
    \end{align}

    \begin{claim}
        Fix $r \in (0, 2 \epsilon]$, $a \in (0, r]$, $\eta \in [0,a/d]$, and $t_*\in\mathbb N$.  For every time $t$ and clique $W_i$, write
        \[
            \widetilde z_{t,i,1}\le\cdots\le\widetilde z_{t,i,d}
        \]
        for the ordered list of the values $\{z_t(w):w\in W_i\}$.  Let $z_0:V\to[0,1]$ satisfy $z_0(x_j)\le H_{r,a,\eta}(x_j)$ for all path vertices and
        \[
            \widetilde z_{0,i,j}\le H_{r,a,\eta}(w_{i,j})\qquad(1\le i\le m,\ 1\le j\le d).
        \]
        Let $\{v_t\}_{1\le t\le t_*}$ be any update sequence, and let $k_i(t)$ be the number of updates in $W_i$ during the first $t$ steps. Suppose the following conditions hold
        \begin{itemize}
            \item[{\bf {(a)}}] $\sum_{\ell = 0}^{d - 1} (\ell \eta)^q \ge (a - d \eta)^q$.  
            \item[{\bf {(b)}}] $\bigl(Ma - \frac{M(d - 1)}{2} \eta - t_* d\eta\bigr) a^{q - 1} \ge \frac{1}{L^q}\bigl((\frac{1}{2} + (r - a))^q - (\frac{1}{2} - (r - a))^q\bigr)$. 
        \end{itemize}
        
        then for any $0 \le t \le t_*$, we have
        \begin{align}
            &z_t(x_j) \le H_{r, a, \eta}(x_j) \qquad(-L \le j \le L), \nonumber \\
            &\widetilde z_{t,i,j} \le \min\bigl(H_{r, a, \eta}(x_0),\ \tfrac12-r+(j-1+k_i(t))\eta\bigr) \qquad(1\le j\le d).
        \end{align}

        Moreover, at time $t_*$ the same pathwise and order-statistic domination holds with $r$ replaced by $r-a$.
    \end{claim}
    \begin{proof}
        We perform induction on $0 \le t \le t_*$.  The case $t=0$ is exactly the hypothesis.  Suppose now that the claim holds for $z_{t-1}$; we prove it at time $t\ge1$.
        
        {\bf Case 1: $v_t \in \{ x_{-L + 1}, \cdots, x_{-1}, x_1, \cdots, x_{L - 1}\}$}
        
            Suppose $v_t = x_i$. Since $z_{t - 1}(x_{i - 1}) \le H_{r, a, \eta}(x_{i - 1})$ and $z_{t - 1}(x_{i + 1}) \le H_{r,a, \eta}(x_{i + 1})$, we have
            \begin{align}
                z_t(x_i) &= \frac{1}{2}(z_{t - 1}(x_{i - 1}) + z_{t - 1}(x_{i + 1})) \nonumber \\ &\le \frac{1}{2}(H_{r, a, \eta}(x_{i - 1}) + H_{r, a, \eta}(x_{i + 1})) \nonumber \\ &= H_{r, a, \eta}(x_i) .
            \end{align}
            
        {\bf Case 2: $v_t \in W_1 \cup \cdots \cup W_m$}

            Suppose $v_t\in W_i$.  Relabel the vertices of $W_i$ temporarily as $u_1,\ldots,u_d$ so that
            \[
                z_{t-1}(u_1)\le\cdots\le z_{t-1}(u_d),
            \]
            and write $v_t=u_j$.  Since $z_{t-1}(x_0)$ and all clique values are at most $H_{r,a,\eta}(x_0)$, the updated value is also at most $H_{r,a,\eta}(x_0)$.

            It remains to prove
            \[
                z_t(v_t)\le \frac12-r+(k_i(t-1)+d)\eta.
            \]
            Assuming this, the order-statistic bound follows explicitly.  For $1\le\ell<d$, changing one entry gives
            \[
                \widetilde z_{t,i,\ell}\le \widetilde z_{t-1,i,\ell+1}
                \le\frac12-r+(\ell+k_i(t-1))\eta.
            \]
            For $\ell=d$, both the old maximum and the newly updated value are at most $\frac12-r+(d+k_i(t-1))\eta$.  Since $k_i(t)=k_i(t-1)+1$, these two estimates are exactly
            \[
                \widetilde z_{t,i,\ell}\le\frac12-r+(\ell-1+k_i(t))\eta
                \qquad(1\le\ell\le d).
            \]

            Set $y_0:=\frac12-r+(k_i(t-1)+d)\eta$.  Since $y\mapsto \frac{d}{dy}L_{z_{t-1}}^{v_t}(y)$ is increasing, it suffices to show that its value at $y_0$ is nonnegative.  By the induction hypothesis,
            \[
                z_{t-1}(x_0)\le\frac12-(r-a),
                \qquad
                z_{t-1}(u_\ell)\le\frac12-r+(k_i(t-1)+\ell-1)\eta.
            \]
            Hence the gap from $y_0$ to $u_\ell$ is at least $(d-\ell+1)\eta$.  Omitting the single term corresponding to $u_j$ from the $d$ numbers $\eta^q,(2\eta)^q,\ldots,(d\eta)^q$ leaves a sum at least $\sum_{s=1}^{d-1}(s\eta)^q$.  Therefore
            \begin{align}
                \frac{d}{dy}L_{z_{t-1}}^{v_t}(y_0)
                &\ge p\left(U(d\eta-a)+\sum_{s=1}^{d-1}(s\eta)^q\right)\\
                &=p\left(\sum_{s=0}^{d-1}(s\eta)^q-(a-d\eta)^q\right)\ge0,
            \end{align}
            where we used $d\eta\le a$ and condition {\bf(a)}.

        {\bf Case 3: $v_t = x_0$}

            As is shown in Case 2, it suffices to show that $\frac{d}{dy}L_{z_{t - 1}}^{v_t}(\frac{1}{2} - (r - a)) \ge 0$.  

            We set $y_0 := \frac{1}{2} - (r - a)$. Note that we have
            \begin{align} \label{eq: d/dy L in case 3}
                \frac{d}{dy} L_{z_{t - 1}}^{v_t}(y_0) &= p\Bigl(U(y_0 - z_{t - 1}(x_1)) + U(y_0 - z_{t - 1}(x_{-1})) \nonumber \\ &+ \sum_{\substack{1 \le i \le m \\ 1 \le j \le d}} U(y_0 - \widetilde z_{t-1,i,j})\Bigr) .
            \end{align}

            By induction hypothesis, we have \[z_{t - 1}(x_1) \le H_{r, a, \eta}(x_1) , \quad z_{t - 1}(x_{-1}) \le H_{r, a, \eta}(x_{-1}) .\]

            Moreover, we have \[ \widetilde z_{t-1,i,j} \le \min(y_0, \frac{1}{2} - r + (j - 1) \eta + k_i(t - 1) \eta) ,\]

            which implies that
            \begin{align}
                U(y_0 - \widetilde z_{t-1,i,j}) &= (y_0 - \widetilde z_{t-1,i,j})^q \nonumber \\ &\ge (a - (j - 1) \eta - k_i(t - 1) \eta)a^{q - 1} .
            \end{align}
            Indeed, if the linear factor on the right is nonpositive the inequality is immediate; if it is positive, then it is at most $a$, and $x^q\ge xa^{q-1}$ for $0\le x\le a$ because $0<q<1$.

            Thus \eqref{eq: d/dy L in case 3} gives that
            \begin{align}
                \frac{d}{dy} L_{z_{t - 1}}^{v_t}(y_0) &\ge p \Bigl( -\frac{1}{L^q}\bigl(\frac{1}{2} + (r - a)\bigr)^q + \frac{1}{L^q}\bigl(\frac{1}{2} - (r - a)\bigr)^q \nonumber \\ &+ \sum_{\substack{1 \le i \le m \\ 1 \le j \le d}} (a - (j - 1) \eta - k_i(t - 1) \eta) a^{q - 1}\Bigr) \nonumber \\ &\ge p \Bigl( (Ma - \frac{M(d - 1)}{2} \eta - t_*d\eta)a^{q - 1} \nonumber \\ &- \frac{1}{L^q}(\frac{1}{2} + (r - a))^q + \frac{1}{L^q}(\frac{1}{2} - (r - a))^q\Bigr) \ge 0 ,
            \end{align}

            where we use the condition {\bf (b)} in the last step. 

        Finally,
        \begin{align*}
            z_{t_*}(x_j)&\le H_{r,a,\eta}(x_j)\le H_{r-a,a,\eta}(x_j),\\
            \widetilde z_{t_*,i,j}&\le H_{r,a,\eta}(x_0)=\frac12-(r-a)\le H_{r-a,a,\eta}(w_{i,j}).
        \end{align*}
        Thus the same pathwise and order-statistic domination holds with $r$ replaced by $r-a$.
    \end{proof}

    We now return to the proof of Lemma~\ref{lem: joint pinned reservoir}.  

    We choose $c_{1,p}:=p^{-1/q} \in (0,1)$. Then, for every $d \in \Nstar$, \[\sum_{\ell=0}^{d-1}\ell^q \ge \int_0^{d-1}x^q\,dx = \frac{(d-1)^p}{p} = c_{1,p}^q(d-1)^p.\] 
    
    We also choose $C_{1, p} > 1$ such that for $r \in (0, 2 \epsilon]$ and $a \in (0, r]$, we always have
    \begin{align}
        \Bigl(\frac{1}{2} + (r - a)\Bigr)^q - \Bigl(\frac{1}{2} - (r - a) \Bigr)^q \le C_{1, p} \epsilon .
    \end{align}

    We set
    \begin{align*}
        K_p &:= \left( \frac{8^{p/q} C_{1, p}}{c_{1, p}}\right)^{1/q} , \\ 
        \eta &:= K_p\epsilon^{1/q} d^{-p/q} L^{-1} M^{-1/q} , \\ a &:= (d + c_{1, p} (d - 1)^{p/q}) \eta , \quad t_* := \left\lfloor \frac{C_{1, p} \epsilon}{L^q a^{q - 1} d \eta} \right\rfloor . 
    \end{align*}

    We will verify conditions {\bf {(a)}} and {\bf {(b)}} for $r \in [a, 2\epsilon]$. Condition {\bf {(a)}} follows from \[ \sum_{\ell = 0}^{d - 1} (\ell \eta)^q \ge c_{1, p}^q (d - 1)^p \eta^q = (a - d \eta)^q, \]

    From the definition of $a$ and $0<c_{1,p}<1$ we will repeatedly use
    \begin{align}\label{eq: a eta comparison}
        c_{1,p}\bigl(1+(d-1)^{p/q}\bigr)\eta
        \le a-\frac{d-1}{2}\eta,
        \qquad
        a\le 2d^{p/q}\eta.
    \end{align}

    Note that we also have
    \begin{align} \label{eq: derivation of condition b}
        \bigl(Ma - \frac{M(d - 1)}{2} \eta - t_* d \eta\bigr)a^{q - 1} &= M\bigl(a - \frac{d - 1}{2} \eta\bigr) a^{q - 1} - t_*d \eta a^{q - 1} \nonumber \\ &\ge c_{1, p}M\eta(1 + (d - 1)^{p/q})(2d^{p/q} \eta)^{q - 1} - \frac{C_{1, p} \epsilon}{L^q} .
    \end{align}

    Note that we have $1 + (d - 1)^{p/q} \ge (d/2)^{p/q}$. By definition of $\eta$, we have
    \begin{align}
        c_{1, p} M \eta(1 + (d - 1)^{p/q})(2d^{p/q} \eta)^{q - 1} &\ge 2^{-p/q} \cdot 2^{q - 1}c_{1, p} M (d^{p/q} \eta)^q \nonumber \\ &= 2^{-p/q} \cdot 2^{q - 1} c_{1, p} M \left( K_p^q\epsilon\right) \cdot L^{-q} M^{-1} \nonumber \\ &\ge \frac{2C_{1, p} \epsilon}{L^q}, 
    \end{align}

    Combining with \eqref{eq: derivation of condition b}, we deduce that
    \begin{align}
        (Ma - \frac{M(d - 1)}{2} \eta - t_*d \eta)a^{q - 1} &\ge \frac{2C_{1, p} \epsilon}{L^q} - \frac{C_{1, p} \epsilon}{L^q} \nonumber \\ &= \frac{C_{1, p} \epsilon}{L^q} \nonumber \\ &\ge \frac{1}{L^q}\bigl((\frac{1}{2} + (r - a))^q - (\frac{1}{2} - (r - a))^q\bigr) . 
    \end{align}
    
    We set $r_k=2\epsilon-ak$ for $k\le\lfloor\epsilon/a\rfloor-1$.  The initial profile $z_0=1-f_0$ satisfies the pathwise and order-statistic hypothesis of the claim with $r=r_0$.  Applying the claim successively block by block shows that the same hypothesis holds at time $kt_*$ with $r=r_k$.

    More is true inside each block.  For every $0\le s\le t_*$, the claim applied to the $k$th block gives
    \[
        \widetilde z_{kt_*+s,i,1}\le H_{r_k,a,\eta}(x_0)=\frac12-(r_k - a)
        \qquad(1\le i\le m).
    \]
    Hence at every such time there is a clique vertex $w$ with
    \[
        f_{kt_*+s}(w)-h(w)=\frac12-z_{kt_*+s}(w)\ge r_k - a = 2 \epsilon - a(k + 1)>\epsilon.
    \]
    Thus we have $\lVert f_{k t_* + s} - h \rVert_\infty > \epsilon$ for every $0 \le k < \lfloor \epsilon/a \rfloor - 1$ and $0 \le s \le t_*$, and
    \[
        \tau_p(\epsilon)\ge t_*\left(\left\lfloor\frac\epsilon a\right\rfloor-1\right).
    \]

    For sufficiently large $L,M$ (with the threshold depending only on the fixed $p$), \eqref{eq: a eta comparison} and the definition of $\eta$ give $a\le\epsilon/4$.  Also,
    \begin{align*}
        L^qa^{q-1}d\eta
        &\le C_pL^qd^{\frac{p(q-1)}q+1}\eta^q\\
        &\le C_p\epsilon M^{-1}d^{\frac{p(q-1)}q+1-p}
         =C_p\epsilon M^{-1}d^{-1/q},
    \end{align*}
    where
    \[
        \frac{p(q-1)}q+1-p=-\frac1q.
    \]
    Consequently, after increasing the fixed-$p$ threshold if necessary, we can assume that 
    \[
        t_*\ge c_{2,p}Md^{1/q} \ge 1.
    \]
    Thus the block length used above is positive in the parameter range to which the lemma is applied.
    Together with $a\le2d^{p/q}\eta$, this implies
    \begin{align}
        \tau_p(\epsilon) &\ge c_{2,p}Md^{1/q}\cdot\frac{\epsilon}{2a}\nonumber \\
        &\ge \frac{c_{2,p}Md^{1/q}}{2}\cdot\frac{\epsilon}{2d^{p/q}\eta}\nonumber \\
 &\ge \frac{c_{2, p}}{8K_p} LM^{1 + 1/q} d^{1/q} \epsilon^{-\rho}. 
    \end{align}
\end{proof}

\begin{rem}
All constants in Lemma~\ref{lem: joint pinned reservoir} are for fixed $p\in(1,2)$.  No uniformity as $p\downarrow1$ is asserted; in particular the threshold $C_p$ may grow rapidly as $p$ approaches~1.  Crucially for Theorem~\ref{th: joint lower p<2}, $C_p$ is independent of $\epsilon$.
\end{rem}

\begin{thm}\label{th: joint lower p<2}
    Fix $p\in(1,2)$. There exist constants $c_p,C_p>0$ such that, for every $N\ge C_p$, every $2\le D\le N$, and every $\epsilon\in(0,1/2]$, there is a connected graph $G=(V=I\bigsqcup B,E)$ with $|V|\le N$ and $D_G\le D$, and an initial profile $f_0:V\to[0,1]$, such that, for every sequence of update vertices in $I$,
    \begin{align}\label{eq: proved joint lower}
        \tau_p(\epsilon)
        \ge c_pN^{\beta_p}\left(\frac{D}{N}\right)^{\theta_p}
        \epsilon^{-\frac{2-p}{p-1}}.
    \end{align}
\end{thm}

\begin{proof}
    First suppose that $\epsilon \le 1/16$. We take
    \begin{align*}
        L := \lfloor N/8 \rfloor, d := \min(\lfloor D - 1 \rfloor, \lfloor N/8 \rfloor) ,
    \end{align*}    
    
    and let $M := md$ be the largest multiple of $d$ not exceeding $N/2$.  
    
    The fact that $d \le N / 8$ implies $M \ge N/2 - d \ge 3N/8$, and $D\le N$ implies $d \ge D/16$ for large $N$.

    The graph $G_{L,m,d}$ has
    \begin{align*}
        \lvert V \rvert = 2L + 1 + M, \qquad 2 \lvert E \rvert = 4L + M(d + 1) .
    \end{align*}
    
    Thus $\lvert V \rvert \le N$. Moreover, $D \ge 2$ and $d + 1 \le D$ give \[ 2 \lvert E\rvert = 4L + M(d + 1) \le D \lvert V \rvert , \] 
    
    Thus $D_G\le D$. 
    
    Lemma~\ref{lem: joint pinned reservoir} gives
    \begin{align*}
        \tau_p(\epsilon) &\ge c_p L M^{1 + 1/q}d^{1/q}\epsilon^{-\rho}\\ &\ge c_p N^{2+1/q} D^{1/q} \epsilon^{-\rho}
         = c_p N^{\beta_p} \left(\frac{D}{N}\right)^{\theta_p} \epsilon^{-\frac{2-p}{p-1}} .
    \end{align*}

    For $\epsilon\in(1/16,1/2]$, use Theorem~\ref{th: with boundary lower bound} instead. On this interval the factor
    $\epsilon^{-(2-p)/(p-1)}$ is bounded above and below by positive constants depending only on $p$, so the same conclusion follows.
\end{proof}

\section{\texorpdfstring{Graph-dependent convergence rate for $p=2$}{Graph-dependent convergence rate for p=2}}  \label{p=2 case}
 
In this section, we still assume that $G=(V,E)$ is a finite connected graph and that $V=I\bigsqcup B$, where $I,B\neq\varnothing$. Write $\lvert V\rvert=n$ and $\lvert I\rvert=n_*$. Let $A_I\in\R^{I\times I}$ be the adjacency matrix restricted to $I$, and let $D=\operatorname{diag}(\deg(v):v\in I)$. Define
\begin{align*}
    Q:=D^{-1}A_I
    \qquad\text{and}\qquad
    S:=D^{-1/2}A_ID^{-1/2}=D^{1/2}QD^{-1/2}.
\end{align*}
Thus, $Q$ is the transition kernel of simple random walk killed upon entering $B$. It is not symmetric in general, whereas $S$ is symmetric and similar to $Q$. Every connected component of the graph induced by $I$ has an edge to $B$, so the killed walk is absorbed almost surely and $\rho(Q)<1$, where $\rho(\cdot)$ is the spectral radius of the matrix. Since $S$ is real symmetric and entrywise nonnegative, the Perron--Frobenius theorem identifies its largest eigenvalue with $\rho(S)=\rho(Q)$.  Set
\begin{align*}
    \lambda:=\rho(Q)=\rho(S)
    \qquad\text{and}\qquad
    \gamma:=1-\lambda>0.
\end{align*}
The symbol $\lambda$ in this section denotes only this killed-walk spectral radius and is unrelated to the Poincar\'e constants used in Subsection~\ref{s:useful inequalities}.

We abbreviate $\en_2(g)$ to $\en(g)$. If $f\vert_B=0$, then, writing $x=f\vert_I$ and $\phi=D^{1/2}x$, we have
\begin{align*}
    \en(f)=x^{\mathsf T}(D-A_I)x
    =\phi^{\mathsf T}(\operatorname{Id}-S)\phi,
    \qquad
    \lVert f\rVert_2^2=x^{\mathsf T}Dx=\lVert\phi\rVert_2^2.
\end{align*}
Consequently,
\begin{align} \label{eqn: variational formula}
    \gamma=\min_{f\vert_B=0, f \neq 0} \frac{\en(f)}{\lVert f \rVert_2^2} ,
\end{align}
where, as before, $\lVert f\rVert_2^2=\sum_{v\in V}\deg(v)f(v)^2$. More precisely, an eigenvector $\phi$ of $S$ with eigenvalue $\lambda$ corresponds to the eigenfunction $f\vert_I=D^{-1/2}\phi$ of $Q$, and conversely $\phi=D^{1/2}f\vert_I$.

\begin{prop}\label{graph dependent bound}
   Consider a connected graph $G=(V,E)$ with $\lvert V \rvert=n$ and a decomposition $V = I \bigsqcup B $ such that $I, B \neq \varnothing$. Suppose that  $n_*:=\lvert I \rvert  \ge 2$.
   \begin{itemize}
        \item[{\bf(a)}]
   For any initial profile $f_0:V \to [0,1]$ and every $\epsilon \in (0,1)$, the dynamics \eqref{eq:pdef} satisfy $\E[\tau_2(\eps)] \le \frac{4n_{*}}{\gamma} \log(n/\epsilon)$.  
     \item[{\bf(b)}] Let $f_0$ equal $1$ on $I$ and $0$ on $B$, and let $t_B$ be as in \eqref{defhit}. For every $\epsilon\in(0,1)$,
    \begin{equation}\label{eqn: direct hitting converse}
        \E_{f_0}[\tau_2(\epsilon)]
        \ge \max\left\{\frac{(1-\epsilon)n_*t_B}{2},\ n_*H_{n_*}\right\}.
    \end{equation}
    In particular,
    \begin{align*}
        \sup_{f_0:V\to[0,1]}\E_{f_0}[\tau_2(1/2)]
        &\ge \max\left\{\frac{n_*t_B}{4},\ n_*H_{n_*}\right\}\\
        &\ge \max\left\{\frac{n_*}{4\gamma},\ n_*H_{n_*}\right\}.
    \end{align*}
\end{itemize}    
\end{prop}
\begin{proof}
    {\bf(a)} Let $h$ be the harmonic extension of $f_0\vert_B$. For an arbitrary current profile $f$ with the same boundary values, set $g:=f-h$, so that $g\vert_B=0$. For $v\in I$, put $m_v(f):=\deg(v)^{-1}\sum_{w\sim v}f(w)$. If $f^{\sharp v}$ is obtained by updating $v$, then only the edges incident to $v$ change, and completing the square gives
    \begin{align*}
        \en(f)-\en(f^{\sharp v})
        &=\sum_{w\sim v}\left((f(v)-f(w))^2-(m_v(f)-f(w))^2\right)\\
        &=\deg(v)\big(f(v)-m_v(f)\big)^2.
    \end{align*}
    Averaging over the uniformly chosen update vertex gives
    \begin{align}\label{eqn:energy decrease}
        \E[\en(f)-\en(f^{\sharp})]
        =\frac1{n_*}\sum_{v\in I}\deg(v)\big(f(v)-m_v(f)\big)^2.
    \end{align}

    Since $h$ is harmonic in $I$ and $g\vert_B=0$,
    \begin{align*}
        \sum_{\{v,w\}\in E}(g(v)-g(w))(h(v)-h(w))
        =\sum_{v\in V}g(v)\sum_{w\sim v}(h(v)-h(w))=0.
    \end{align*}
    Consequently,
    \begin{equation}\label{eqn of 2-energy}
        \en(f)-\en(h)=\en(g).
    \end{equation}
    Harmonicity of $h$ also gives, for every $v\in I$,
    \begin{align}\label{eqn:har}
        f(v)-m_v(f)=g(v)-\frac1{\deg(v)}\sum_{w\sim v}g(w).
    \end{align}
    If $g=0$, every update leaves $f=h$ unchanged, so the desired one-step contraction is immediate. If $g\neq0$, then \eqref{eqn:energy decrease}, \eqref{eqn:har}, and the Cauchy--Schwarz inequality imply
    \begin{align*}
        \lVert g\rVert_2^2\E[\en(f)-\en(f^{\sharp})]
        &\ge\frac1{n_*}\left(\sum_{v\in I}g(v)\left(\deg(v)g(v)-\sum_{w\sim v}g(w)\right)\right)^2\\
        &=\frac1{n_*}\en(g)^2.
    \end{align*}
    Therefore, by \eqref{eqn: variational formula} and \eqref{eqn of 2-energy},
    \begin{align*}
        \E[\en(f)-\en(f^{\sharp})]
        \ge\frac{\gamma}{n_*}\big(\en(f)-\en(h)\big).
    \end{align*}

    Apply this estimate conditionally at each step and write $R_t:=\en(f_t)-\en(h)$. Since $R_0\le\en(f_0)\le\lvert E\rvert$, we obtain
    \begin{align*}
        \E[R_t]
        \le\left(1-\frac\gamma{n_*}\right)^tR_0
        \le\lvert E\rvert\exp\left(-\frac{\gamma t}{n_*}\right).
    \end{align*}

    For every $u\in I$, choose a simple path $u=v_0\sim v_1\sim\cdots\sim v_k\in B$. Since $g(v_k)=0$ and $k\le n-1$,
    \begin{align*}
        g(u)^2
        &=\left(\sum_{j=0}^{k-1}\bigl(g(v_j)-g(v_{j+1})\bigr)\right)^2\\
        &\le k\sum_{j=0}^{k-1}(g(v_j)-g(v_{j+1}))^2
        \le(n-1)\en(g).
    \end{align*}
    An averaging update cannot increase $\lVert g_t\rVert_\infty$. Thus, using the strict inequality dictated by the definition of $\tau_2(\eps)$,
    \begin{align*}
        \{\tau_2(\eps)>t\}=\{\lVert g_t\rVert_\infty>\eps\}
        \subseteq\left\{R_t>\frac{\eps^2}{n-1}\right\}.
    \end{align*}
    Set $a:=\gamma/n_*$ and $C:=(n-1)\lvert E\rvert/\eps^2$. Markov's inequality yields
    \begin{align*}
        \Pr(\tau_2(\eps)>t)\le\min\{1,Ce^{-at}\}.
    \end{align*}
    Here $C\ge1$ (indeed $n\ge3$ and $\epsilon<1$). Since $\tau_2(\eps)$ is integer-valued and $x\mapsto\min\{1,Ce^{-ax}\}$ is nonincreasing,
    \begin{align*}
        \E[\tau_2(\eps)]
        &=\sum_{t=0}^\infty\Pr(\tau_2(\eps)>t)\\
        &\le1+\int_0^\infty\min\{1,Ce^{-ax}\}\,\mathrm{d}x
        =1+\frac{\log C+1}{a}.
    \end{align*}
    Here $n\ge3$, $0<a\le1/2$, and $C\le n(n-1)^2/(2\eps^2)$. Hence
    \begin{align*}
        4\log(n/\eps)-(a+\log C+1)
        &\ge\log\left(\frac{2n^3}{(n-1)^2}\right)+2\log(1/\eps)-1-a\\
        &\ge\log(27/2)-\frac32>0.
    \end{align*}
    This proves $\E[\tau_2(\eps)]\le4n_*\gamma^{-1}\log(n/\eps)$.

    {\bf(b)} Start the dynamics from $f_0=\mathbf1_I$, whose harmonic extension is $h=0$. Define
    \begin{align*}
        \widehat Q:=\left(1-\frac1{n_*}\right)\operatorname{Id}+\frac1{n_*}Q.
    \end{align*}
    This is the transition kernel, restricted to $I$, of a walk $(Y_t)$ that at each step stays put with probability $1-1/n_*$ and otherwise takes a simple random walk step, and is absorbed on entering $B$. Denote its absorption time by $\sigma_B$, and let $\mathbf1$ be the constant-one vector on $I$. Conditioning on the uniformly chosen update vertex and iterating gives
    \begin{align}\label{eqn: mean profile survival}
        \E_{f_0}[f_t\vert_I]&=\widehat Q^{\,t}\mathbf1,\nonumber\\
        \E_{f_0}[f_t(u)]&=\Pr_u(\sigma_B>t)
        \quad(u\in I).
    \end{align}

    Write $a(u):=\E_u[\sigma_B]$ for $u\in I$, and set $a=0$ on $B$. Since $\rho(Q)<1$, the mean absorption-time vector satisfies
    \begin{align*}
        a\vert_I=(\operatorname{Id}-\widehat Q)^{-1}\mathbf1
        =n_*(\operatorname{Id}-Q)^{-1}\mathbf1.
    \end{align*}
    Consequently, $T:=\max_{u\in I}a(u)=n_*t_B$. Choose $u\in I$ with $a(u)=T$. The Markov property at a deterministic integer time $t\ge0$ yields
    \begin{align*}
        T
        &=\E_u[\sigma_B\wedge t]
          +\E_u\!\left[\mathbf1_{\{\sigma_B>t\}}a(Y_t)\right]\\
        &\le t+T\Pr_u(\sigma_B>t).
    \end{align*}
    Thus \eqref{eqn: mean profile survival} implies
    \begin{equation}\label{eqn: mean profile lower bound}
        \E_{f_0}[f_t(u)]\ge1-\frac{t}{T}.
    \end{equation}

    For every random variable $X\in[0,1]$ and $\epsilon\in(0,1)$,
    \begin{align*}
        \Pr(X>\epsilon)\ge\frac{\E[X]-\epsilon}{1-\epsilon}.
    \end{align*}
    All values of $f_t$ lie in $[0,1]$, and an averaging update cannot increase $\lVert f_t\rVert_\infty$. Hence, writing $x_+:=\max\{x,0\}$, we obtain from \eqref{eqn: mean profile lower bound}
    \begin{align*}
        \Pr_{f_0}(\tau_2(\epsilon)>t)
        &\ge\Pr_{f_0}(f_t(u)>\epsilon)\\
        &\ge\left(1-\frac{t}{(1-\epsilon)T}\right)_+.
    \end{align*}
    Summing over integer times and comparing the nonincreasing summand with its integral gives
    \begin{align*}
        \E_{f_0}[\tau_2(\epsilon)]
        &\ge\sum_{t=0}^{\infty}\left(1-\frac{t}{(1-\epsilon)T}\right)_+\\
        &\ge\int_0^{(1-\epsilon)T}\left(1-\frac{s}{(1-\epsilon)T}\right)\,\mathrm{d}s
        =\frac{(1-\epsilon)n_*t_B}{2}.
    \end{align*}

    The same initial profile gives the coupon-collector term: until every interior vertex has been selected at least once, some unupdated vertex still has value $1$, so $\tau_2(\epsilon)$ has not occurred for any $\epsilon<1$. The expected time to select all $n_*$ interior vertices is $n_*H_{n_*}$. Combining the two estimates proves \eqref{eqn: direct hitting converse}. Taking $\epsilon=1/2$ gives the first inequality in the final display of part {\bf(b)}. The second follows from $t_B\ge\gamma^{-1}$, proved independently in Proposition~\ref{hitting time estimate} below.

\end{proof}

\begin{lem}\label{lem: geometric survival sum}
    For every integer $m\ge2$ and every $\lambda\in(0,1)$, define
    \begin{align*}
        R_m(\lambda):=1+\sum_{j=0}^\infty\min\{1,m\lambda^j\}.
    \end{align*}
    Then
    \begin{align*}
        (1-\lambda)R_m(\lambda)
        < \log(4m).
    \end{align*}
\end{lem}
\begin{proof}
    Put $L:=\log m$, $a:=-\log\lambda$, $\delta:=1-\lambda$, and
    $k:=\lceil L/a\rceil$. If $z:=e^{L-ka}=m\lambda^k$, then
    $\lambda<z\le1$ and the geometric sum gives
    \begin{align*}
        \delta R_m(\lambda)=\delta(1+k)+z,
        \qquad
        \delta R_m(\lambda)-L=k(\delta-a)+\delta+z-\log z.
    \end{align*}
    Since $\delta\le a$ and $z\mapsto z-\log z$ is decreasing on $(0,1]$, the case $k\ge2$ gives
    \begin{align*}
        \delta R_m(\lambda)-L
        \le2(\delta-a)+\delta+\lambda+a
        =3-2\lambda+\log\lambda
        \le2-\log2.
    \end{align*}
    If $k=1$, then $2\lambda\le z\le1$, and the same argument gives
    \begin{align*}
        \delta R_m(\lambda)-L
        \le2\delta-a+2\lambda-\log(2\lambda)
        =2-\log2.
    \end{align*}
    Both imply the claimed estimate in Lemma~\ref{lem: geometric survival sum}. 
\end{proof}

\begin{rem}
The additive constant in this estimate is sharp: in fact
\[
    \sup_{m\ge2,\ 0<\lambda<1}
    \bigl((1-\lambda)R_m(\lambda)-\log m\bigr)=2-\log2,
\]
and the supremum is approached when $m=2$ and $\lambda\downarrow0$.
\end{rem}

\begin{prop}\label{hitting time estimate}
    Assume $n_*=\lvert I\rvert\ge2$. Let $t_{B}$ be the maximal hitting time, i.e.,
    \begin{eqnarray*}
        t_B:=\max_{v\in I}\E_v[\tau_B],
    \end{eqnarray*}
    where $\E_v$ represents expectation for the simple random walk on $G$ starting at $v$  and $\tau_B$ is the hitting time of the boundary set $B$. Then, 
    \begin{align}
        \frac{1}{\gamma}\le t_B<\frac{\log(4n_*)}{\gamma} .
    \end{align}
\end{prop}
\begin{proof}
    Let
    \begin{align*}
        \mathcal G:=(\operatorname{Id}-Q)^{-1}=\sum_{k=0}^\infty Q^k
    \end{align*}
    be the Green kernel of the walk killed upon entering $B$. For $u,v\in I$, $\mathcal G(u,v)$ is the expected number of visits to $v$ before $\tau_B$ by a walk started at $u$. Thus
    \begin{align}\label{hitting time}
        \E_u[\tau_B]=\sum_{v\in I}\mathcal G(u,v).
    \end{align}

    For the lower bound, we choose a nonnegative right Perron eigenvector $r$ of $Q$, normalize it by $\lVert r\rVert_\infty=1$, and choose $u$ with $r(u)=1$. Since $0\le r\le\mathbf 1$ and $\mathcal G$ has nonnegative entries, we have
    \begin{align*}
        \E_u[\tau_B]
        =(\mathcal G\mathbf 1)(u)
        \ge(\mathcal G r)(u)
        =\frac{r(u)}{1-\lambda}
        =\frac1\gamma.
    \end{align*}

    For the upper bound, write $d(v):=\deg(v)$ and equip functions on $I$ with the weighted norm
    \begin{align*}
        \lVert z\rVert_{2,d}^2:=\sum_{v\in I}d(v)z(v)^2.
    \end{align*}
    If $f\vert_B=0$ and $z=f\vert_I$, then $\lVert z\rVert_{2,d}=\lVert f\rVert_2$ in the weighted notation of \eqref{alphanorm}.
    The kernel $Q$ is self-adjoint for the corresponding weighted inner product, and the isometry $z\mapsto D^{1/2}z$ identifies it with the symmetric matrix $S$.  By the spectral theorem for this self-adjoint operator,
    $\lVert Q\rVert_{2,d}=\max_j\lvert\mu_j\rvert=\rho(Q)=\lambda$, where the $\mu_j$ are the eigenvalues of $S$. Let $\mathbf1$ denote the constant-one function on $I$, put $q:=Q\mathbf1$, and write $d_I(v)$ for the number of neighbors of $v$ in $I$. Then
    \begin{align*}
        q(v)=\frac{d_I(v)}{d(v)}
        \quad\text{and}\quad
        \lVert q\rVert_{2,d}^2
        =\sum_{v\in I}\frac{d_I(v)^2}{d(v)}
        \le\sum_{v\in I}d_I(v)
        \le n_*(n_*-1).
    \end{align*}
    Thus, for every $u\in I$ and integer $r\ge0$,
    \begin{align*}
        \Pr_u(\tau_B>r+1)
        &=(Q^{r+1}\mathbf1)(u)=(Q^rq)(u)\\
        &\le\frac{\lVert Q^rq\rVert_{2,d}}{\sqrt{d(u)}}
        \le\lambda^r\lVert q\rVert_{2,d}
        \le n_*\lambda^r.
    \end{align*}
    If $\lambda=0$, then symmetry of $S$ implies $S=0$, hence $Q=0$ and $t_B=1<\log(4n_*)/\gamma$. Suppose now that $0<\lambda<1$. Using also the trivial bound by $1$ and Lemma~\ref{lem: geometric survival sum}, we obtain, for every $u\in I$,
    \begin{align*}
        \E_u[\tau_B]
        &=1+\sum_{r=0}^\infty\Pr_u(\tau_B>r+1)\\
        &\le1+\sum_{r=0}^\infty\min\{1,n_*\lambda^r\}
        <\frac{\log(4n_*)}{1-\lambda}
        =\frac{\log(4n_*)}{\gamma}.
    \end{align*}
    Taking the maximum over $u$ proves the strict upper bound.
\end{proof}

\begin{rem}\label{prop: sharpness of hitting comparison}
    The two sides of Proposition~\ref{hitting time estimate} have the following sharpness properties.
    \begin{itemize}
        \item[{\bf (a)}] Let $G$ be the path on $\{0,1,\ldots,m\}$, with $B=\{0\}$ and $I=\{1,\ldots,m\}$. Then
        \begin{align*}
            t_B=m^2,
            \qquad
            \gamma=1-\cos\left(\frac{\pi}{2m}\right),
            \qquad
            1\le\gamma t_B<\frac{\pi^2}{8}.
        \end{align*}
        \item[{\bf (b)}] Fix an integer $b\ge2$. Let $T_{b,h}$ be the full rooted $b$-ary tree of height $h$, with levels $0,\ldots,h$, and take the leaves at level $h$ to be $B$. If $I$ consists of levels $0,\ldots,h-1$, then
        \begin{align*}
            m:=\lvert I\rvert=\frac{b^h-1}{b-1}
        \end{align*}
        and
        \begin{align*}
            t_B
            =\frac{b+1}{b-1}h
             -\frac{2b}{(b-1)^2}(1-b^{-h}).
        \end{align*}
        Moreover, with $c_b:=1-2\sqrt b/(b+1)>0$,
        \begin{align*}
            c_bh\le\gamma t_B\le\frac{b+1}{b-1}h.
        \end{align*}
        Consequently, $\gamma t_B=\Theta_b(\log m)$ as $h\to\infty$.
    \end{itemize}
\end{rem}
\begin{proof}
    For part {\bf (a)}, let $H_i:=\E_i[\tau_B]$. The first-step recurrences are
    \begin{align*}
        H_0=0,
        \qquad
        H_i=1+\frac{H_{i-1}+H_{i+1}}2\quad(1\le i<m),
        \qquad
        H_m=1+H_{m-1}.
    \end{align*}
    Their solution is $H_i=i(2m-i)$, which is increasing in $i$; hence $t_B=H_m=m^2$. For the Perron eigenvalue, impose $x_0=0$ and take $x_i=\sin(i\theta)$. The interior recurrence gives $Qx=\cos(\theta)x$, while the endpoint equation at $m$ reduces to $\cos(m\theta)=0$. The smallest positive choice $\theta=\pi/(2m)$ gives a strictly positive eigenvector, and therefore
    \begin{align*}
        \lambda=\cos\left(\frac{\pi}{2m}\right).
    \end{align*}
    The elementary inequalities
    \begin{align*}
        \frac{4x^2}{\pi^2}\le1-\cos x<\frac{x^2}{2}
        \qquad(0<x\le\pi/2)
    \end{align*}
    now give $1\le\gamma m^2<\pi^2/8$.

    For part {\bf (b)}, symmetry makes the mean hitting time constant on each level; denote its value at level $k$ by $H_k$. Then
    \begin{align*}
        H_h=0,
        \qquad H_0=1+H_1,
        \qquad
        H_k=1+\frac{H_{k-1}+bH_{k+1}}{b+1}\quad(1\le k<h).
    \end{align*}
    Set $\delta_k:=H_k-H_{k+1}$. The preceding recurrences give $\delta_0=1$ and
    \begin{align*}
        \delta_k=\frac{b+1-2b^{-k}}{b-1}>0
        \qquad(0\le k<h).
    \end{align*}
    Thus the root maximizes the mean hitting time, and summing these differences yields
    \begin{align*}
        t_B=H_0
        =\sum_{k=0}^{h-1}\delta_k
        =\frac{b+1}{b-1}h
         -\frac{2b}{(b-1)^2}(1-b^{-h}).
    \end{align*}
    In particular,
    \begin{align}\label{eq: tree hitting comparison}
        h\le t_B\le\frac{b+1}{b-1}h.
    \end{align}

    To bound the spectral gap, define the positive function
    \begin{align*}
        r(v):=b^{-\operatorname{dist}(v,\mathrm{root})/2}
        \qquad(v\in I).
    \end{align*}
    At every interior vertex, direct substitution gives
    \begin{align*}
        Qr(v)\le\frac{2\sqrt b}{b+1}r(v).
    \end{align*}
    This includes the root and level $h-1$, where the inequality is strict. By the Perron-Frobenius theorem,
    \begin{align*}
        \lambda\le\frac{2\sqrt b}{b+1},
        \qquad
        \gamma\ge c_b.
    \end{align*}
    Combining this with \eqref{eq: tree hitting comparison} and $\gamma\le1$ proves the displayed two-sided estimate. Finally, $h=\log_b((b-1)m+1)$, which proves the asserted logarithmic order.
\end{proof}

\begin{proof}[Proof of Proposition~\ref{graph dependent bound: hitting time}]
    The inequality $\gamma^{-1}\le t_B$ from Proposition~\ref{hitting time estimate}, together with Proposition~\ref{graph dependent bound}{\rm(a)}, gives
    \begin{align*}
        \E[\tau_2(\eps)]\le\frac{4n_*}{\gamma}\log(n/\eps)
        \le4n_*t_B\log(n/\eps).
    \end{align*}
    The converse \eqref{hitbound: converse} follows directly from Proposition~\ref{graph dependent bound}{\rm(b)}, using the same initial profile $f_0=\mathbf1_I$ for every $\epsilon\in(0,1)$.
\end{proof}

\begin{rem}
    Remark~\ref{prop: sharpness of hitting comparison} establishes the sharpness of the spectral/hitting-time comparison in Proposition~\ref{hitting time estimate}. The direct argument in Proposition~\ref{graph dependent bound}{\rm(b)} avoids using the upper side of that comparison, and hence yields a hitting-time lower bound with no logarithmic loss on every graph.

    On the fixed-$b$ trees in Remark~\ref{prop: sharpness of hitting comparison}{\rm(b)}, we have $t_B=\Theta_b(\log n_*)$ and $\gamma\ge c_b>0$. Since the total number of vertices is $\Theta_b(n_*)$, Proposition~\ref{graph dependent bound} gives
    \[
        \sup_{f_0:V\to[0,1]}\E_{f_0}[\tau_2(1/2)]
        =\Theta_b(n_*\log n_*)=\Theta_b(n_*t_B).
    \]

    By contrast, let $G$ be a star with $n_*$ interior leaves and its center as the sole boundary vertex. Here $t_B=\gamma=1$. Starting from $f_0=\mathbf1_I$, each leaf becomes zero at its first update, so for every $\epsilon\in(0,1)$,
    \[
        \E_{f_0}[\tau_2(\epsilon)]=n_*H_{n_*}.
    \]
    Thus, at a fixed error threshold, a logarithmic factor is necessary in a general upper bound with prefactor $n_*t_B$. This example does not settle whether the logarithmic factor in the coarser worst-case bound of Theorem~\ref{th:main lpbound}{\rm(b)} can be removed. The optimal graph-dependent approximation time on general graphs remains open.
\end{rem}
\section{Open questions}\label{s:open}

As conjectured in~\cite{noboundarycase} for the $\epsilon$-consensus time of $\ell^p$-energy minimizing dynamics, we also conjecture that the $\epsilon$-approximation time should be approximately monotonically decreasing with respect to $p$. 

\begin{question} 
    For $p \in (1,\infty)$, a finite connected graph $G = (V, E)$, a decomposition $V = I \bigsqcup B$ with $I, B \ne \varnothing$, and $\epsilon \in (0,1/2]$, define \[ T_p(G, B, \epsilon) := \sup_{f_0: V \to [0, 1]} \E_{f_0}[\tau_p(\epsilon)] ,\]

    where $\E_{f_0}[\tau_p(\epsilon)]$ denotes the expected $\epsilon$-approximation time with initial profile $f_0$. 

    Is it true that for every $1 < p_1 < p_2$, there exists a constant $C = C_{p_1, p_2}$ such that $T_{p_2}(G, B, \epsilon) \le C T_{p_1}(G, B, \epsilon)$ for every $\epsilon \in (0, 1/2]$, every finite connected graph $G = (V, E)$ and every decomposition $V = I \bigsqcup B$ with $I, B \neq \varnothing$ ?
\end{question}

\medskip


Another challenge is to improve our bounds for expected $\epsilon$-approximation times, which apply to all graphs with given size and average degree, to sharp graph-dependent bounds.

\begin{question}
    Fix $p \in (1, \infty)$ and a finite connected graph $G = (V, E)$ with decomposition $V = I \bigsqcup B$. Can we find precise graph-dependent bounds for the maximum over initial profiles in $[0,1]^V$ of
    $\E[\tau_p(\epsilon)]$ ? 
    
    For example, if $b\in B$, the commute-time identity and the bound $R_{\rm eff}(v,b)\le \operatorname{dist}(v,b)$ give
    \[
        \E_v[\tau_B]\le\E_v[\tau_b]
        \le \E_v[\tau_b]+\E_b[\tau_v]
        =2|E|R_{\rm eff}(v,b)
        \le2|E|\Diam(G).
    \]
    Hence Proposition~\ref{graph dependent bound: hitting time} implies
    \[
        \E[\tau_2(\epsilon)]\le 8|I|\,|E|\,\Diam(G)\log(n/\epsilon).
    \]
    Does a bound of the same form hold when $\tau_2$ is replaced by $\tau_p$ for $p>2$?
\end{question}

\medskip


\section{Acknowledgment}

Yuval Peres is supported by the National Natural Science Foundation of China under Grant Number W2531011.

\bibliographystyle{plain}
\bibliography{lp_bibliography_revised}

\end{document}